\documentclass[a4paper,10pt]{amsart}

\usepackage{amsmath,amssymb,amsthm,a4wide,graphicx,tikz}
\usetikzlibrary{shapes}

\usepackage{caption}
\newtheorem*{theorem}{Theorem}
\newtheorem{thm}{Theorem}

\newtheorem{lemma}[thm]{Lemma}
\newtheorem*{corollary}{Corollary}
\newtheorem{proposition}[thm]{Proposition}

\newcommand{\argmax}{\operatorname*{argmax}}

\begin{document}

\title[]{Triangles capturing many lattice points}
\keywords{Lattice point counting, shape optimization, limit shapes.}
\subjclass[2010]{11P21, 52B05, 52C05 (primary) and 11H06, 90C27 (secondary)}

\author[]{Nicholas F. Marshall}\address{Program in Applied Mathematics, Yale University, New Haven, CT 06511, USA}
\email{nicholas.marshall@yale.edu}

\author[]{Stefan Steinerberger}
\address{Department of Mathematics, Yale University, New Haven, CT 06511, USA}
\email{stefan.steinerberger@yale.edu}

\begin{abstract} 
We study a combinatorial problem that recently arose in the context of shape
optimization: among all triangles with vertices $(0,0)$, $(x,0)$, and $(0,y)$
and fixed area, which one encloses the most lattice points from
$\mathbb{Z}_{>0}^2$?  Moreover, does its shape necessarily converge to the
isosceles triangle $(x=y)$ as the area becomes large? Laugesen and Liu suggested
that, in contrast to similar problems, there might not be a limiting shape.  We
prove that the limiting set is indeed nontrivial and contains infinitely many
elements. We also show that there exist `bad' areas where no triangle is
particularly good at capturing lattice points and show that there exists an
infinite set of slopes $y/x$ such that any associated triangle captures more
lattice points than any other fixed triangle for infinitely many (and
arbitrarily large) areas; this set of slopes is a fractal subset of $[1/3, 3]$
and has Minkowski dimension at most $3/4$. 
\end{abstract}

\maketitle

\section{Introduction}
\subsection{Introduction} 
In 2012, Antunes \& Freitas \cite{AntunesFreitas2012} proved that among all
axes-parallel ellipses that are centered at the origin and of a fixed area,
the ellipse enclosing the most lattice points from $\mathbb{Z}_{>0}^2$ converges
to the circle as the area becomes large. This problem originally arose in the
study of variational aspects of spectral geometry, more specifically in the
context of minimizing large eigenvalues of the Laplace operator on rectangles.

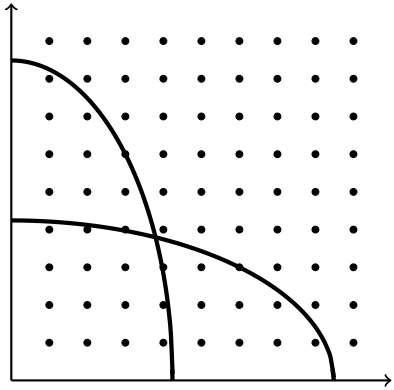
\begin{figure}[h!]
\begin{tikzpicture}[scale=0.5]
\draw [thick, ->]  (0,0) -- (10,0);
\draw [thick, ->]  (0,0) -- (0,10);
\draw[ultra thick,domain=0:4.24265,samples=100,smooth,variable=\x] 
plot ({\x},{sqrt(2)*sqrt(36 - 2*\x*\x))});
\draw [ultra thick] (4.242,0) -- (4.23,0.3);
\draw[ultra thick,domain=0:8.485,samples=100,smooth,variable=\x] 
plot ({\x},{sqrt(2)*sqrt(9 - \x*\x/8))});
\draw [ultra thick] (8.485,0) -- (8.45,0.3);
\filldraw (1,1) circle (0.09cm);
\filldraw (1,2) circle (0.09cm);
\filldraw (1,3) circle (0.09cm);
\filldraw (1,4) circle (0.09cm);
\filldraw (1,5) circle (0.09cm);
\filldraw (1,6) circle (0.09cm);
\filldraw (1,7) circle (0.09cm);
\filldraw (1,8) circle (0.09cm);
\filldraw (1,9) circle (0.09cm);
\filldraw (2,1) circle (0.09cm);
\filldraw (2,2) circle (0.09cm);
\filldraw (2,3) circle (0.09cm);
\filldraw (2,4) circle (0.09cm);
\filldraw (2,5) circle (0.09cm);
\filldraw (2,6) circle (0.09cm);
\filldraw (2,7) circle (0.09cm);
\filldraw (2,8) circle (0.09cm);
\filldraw (2,9) circle (0.09cm);
\filldraw (3,1) circle (0.09cm);
\filldraw (3,2) circle (0.09cm);
\filldraw (3,3) circle (0.09cm);
\filldraw (3,4) circle (0.09cm);
\filldraw (3,5) circle (0.09cm);
\filldraw (3,6) circle (0.09cm);
\filldraw (3,7) circle (0.09cm);
\filldraw (3,8) circle (0.09cm);
\filldraw (3,9) circle (0.09cm);
\filldraw (4,1) circle (0.09cm);
\filldraw (4,2) circle (0.09cm);
\filldraw (4,3) circle (0.09cm);
\filldraw (4,4) circle (0.09cm);
\filldraw (4,5) circle (0.09cm);
\filldraw (4,6) circle (0.09cm);
\filldraw (4,7) circle (0.09cm);
\filldraw (4,8) circle (0.09cm);
\filldraw (4,9) circle (0.09cm);
\filldraw (5,1) circle (0.09cm);
\filldraw (5,2) circle (0.09cm);
\filldraw (5,3) circle (0.09cm);
\filldraw (5,4) circle (0.09cm);
\filldraw (5,5) circle (0.09cm);
\filldraw (5,6) circle (0.09cm);
\filldraw (5,7) circle (0.09cm);
\filldraw (5,8) circle (0.09cm);
\filldraw (5,9) circle (0.09cm);
\filldraw (6,1) circle (0.09cm);
\filldraw (6,2) circle (0.09cm);
\filldraw (6,3) circle (0.09cm);
\filldraw (6,4) circle (0.09cm);
\filldraw (6,5) circle (0.09cm);
\filldraw (6,6) circle (0.09cm);
\filldraw (6,7) circle (0.09cm);
\filldraw (6,8) circle (0.09cm);
\filldraw (6,9) circle (0.09cm);
\filldraw (7,1) circle (0.09cm);
\filldraw (7,2) circle (0.09cm);
\filldraw (7,3) circle (0.09cm);
\filldraw (7,4) circle (0.09cm);
\filldraw (7,5) circle (0.09cm);
\filldraw (7,6) circle (0.09cm);
\filldraw (7,7) circle (0.09cm);
\filldraw (7,8) circle (0.09cm);
\filldraw (7,9) circle (0.09cm);
\filldraw (8,1) circle (0.09cm);
\filldraw (8,2) circle (0.09cm);
\filldraw (8,3) circle (0.09cm);
\filldraw (8,4) circle (0.09cm);
\filldraw (8,5) circle (0.09cm);
\filldraw (8,6) circle (0.09cm);
\filldraw (8,7) circle (0.09cm);
\filldraw (8,8) circle (0.09cm);
\filldraw (8,9) circle (0.09cm);
\filldraw (9,1) circle (0.09cm);
\filldraw (9,2) circle (0.09cm);
\filldraw (9,3) circle (0.09cm);
\filldraw (9,4) circle (0.09cm);
\filldraw (9,5) circle (0.09cm);
\filldraw (9,6) circle (0.09cm);
\filldraw (9,7) circle (0.09cm);
\filldraw (9,8) circle (0.09cm);
\filldraw (9,9) circle (0.09cm);
\end{tikzpicture}
\caption{Among all ellipses with the same area, which captures the most lattice points?}
\end{figure}

Formally, let $R_a$ denote a $1/a \times a$ rectangle, and $\lambda_{1,a}
\leq \lambda_{2,a} \leq \dots$ denote the eigenvalues of the Dirichlet-Laplacian
$-\Delta$ on $R_a$. The explicit form of the eigenvalues allows us to
compute the number of eigenvalues below a certain threshold as the
number of lattice points inside an ellipse: 
$$
\# \left\{k \in \mathbb{Z}_{>0}: \lambda_{k,a} \leq \pi^2 r^2 \right\} = \#
\left\{ (m,n) \in \mathbb{Z}_{>0}^2: (am)^2 + (n/a)^2 \leq r^2 \right\}.
$$

A natural question is now the following: among all rectangles with area 1, which
minimizes the $k$-th eigenvalue of the Dirichlet-Laplacian? If we denote a
sequence of minimizers by $(R_{a_k})$, then the behavior of $(a_k)$ is
rather complicated and not well understood. However, Antunes \& Freitas managed to determine the asymptotic behavior.
\begin{theorem}[Antunes \& Freitas, \cite{AntunesFreitas2012}] We have
$$ \lim_{k \rightarrow \infty}{a_k} = 1.$$
\end{theorem}
This means, geometrically, that the ellipse which encloses the most positive integer lattice points tends towards a circle
as the area tends towards infinity.
This result was quite influential and has inspired several other works in high
frequency shape optimization \cite{AntunesFreitas2016, Berger2015, BucerFreitas2013, vandenBergBucurGittins2016, vandenBergGittins2016} and new lattice point theorems \cite{AriturkLaugesen2017, LaugesenLiu2016}. We especially
emphasize the recent results of Laugesen \& Liu \cite{LaugesenLiu2016} and Ariturk \& Laugesen \cite{AriturkLaugesen2017}, which
motivate this paper. Laugesen \& Liu extended the result of Antunes \& Freitas
to a large family of concave curves (which includes the $\ell^p$-unit balls for $1 < p
< \infty$). Ariturk \& Laugesen \cite{AriturkLaugesen2017} established a similar
result for a large class of decreasing convex curves (which includes the
$\ell^p$-unit balls
for $0 < p < 1$).

\subsection{Triangles capturing points.} The approaches in
\cite{AriturkLaugesen2017,LaugesenLiu2016} do not cover the $p=1$ case
(corresponding to triangles) and the case was left as an open problem:

\begin{quote}
The case $p=1$ remains open, where the question is: which right triangles in the
first quadrant with two sides along the axes will enclose the most lattice
points, as the area tends to infinity?  [...] Our numerical evidence suggests
that the right triangle enclosing the most lattice points in the open first
quadrant (and with right angle at the origin) does not approach a $45-45-90$
degree triangle as $r \rightarrow \infty$. Instead one seems to get an infinite
limit set of optimal triangles. (from \cite{LaugesenLiu2016})
\end{quote}
The purpose of our paper is to prove these statements.
\begin{figure}[h!]
\begin{tabular}{cc}
\begin{tikzpicture}[scale=0.4]
\draw [thick, ->]  (0,0) -- (10,0);
\draw [thick, ->]  (0,0) -- (0,10);
\draw[ultra thick, domain=0:7.13,samples=100,smooth,variable=\x] 
plot ({\x},{7.13-\x)});
\filldraw (1,1) circle (0.09cm);
\filldraw (1,2) circle (0.09cm);
\filldraw (1,3) circle (0.09cm);
\filldraw (1,4) circle (0.09cm);
\filldraw (1,5) circle (0.09cm);
\filldraw (1,6) circle (0.09cm);
\filldraw (1,7) circle (0.09cm);
\filldraw (1,8) circle (0.09cm);
\filldraw (1,9) circle (0.09cm);
\filldraw (2,1) circle (0.09cm);
\filldraw (2,2) circle (0.09cm);
\filldraw (2,3) circle (0.09cm);
\filldraw (2,4) circle (0.09cm);
\filldraw (2,5) circle (0.09cm);
\filldraw (2,6) circle (0.09cm);
\filldraw (2,7) circle (0.09cm);
\filldraw (2,8) circle (0.09cm);
\filldraw (2,9) circle (0.09cm);
\filldraw (3,1) circle (0.09cm);
\filldraw (3,2) circle (0.09cm);
\filldraw (3,3) circle (0.09cm);
\filldraw (3,4) circle (0.09cm);
\filldraw (3,5) circle (0.09cm);
\filldraw (3,6) circle (0.09cm);
\filldraw (3,7) circle (0.09cm);
\filldraw (3,8) circle (0.09cm);
\filldraw (3,9) circle (0.09cm);
\filldraw (4,1) circle (0.09cm);
\filldraw (4,2) circle (0.09cm);
\filldraw (4,3) circle (0.09cm);
\filldraw (4,4) circle (0.09cm);
\filldraw (4,5) circle (0.09cm);
\filldraw (4,6) circle (0.09cm);
\filldraw (4,7) circle (0.09cm);
\filldraw (4,8) circle (0.09cm);
\filldraw (4,9) circle (0.09cm);
\filldraw (5,1) circle (0.09cm);
\filldraw (5,2) circle (0.09cm);
\filldraw (5,3) circle (0.09cm);
\filldraw (5,4) circle (0.09cm);
\filldraw (5,5) circle (0.09cm);
\filldraw (5,6) circle (0.09cm);
\filldraw (5,7) circle (0.09cm);
\filldraw (5,8) circle (0.09cm);
\filldraw (5,9) circle (0.09cm);
\filldraw (6,1) circle (0.09cm);
\filldraw (6,2) circle (0.09cm);
\filldraw (6,3) circle (0.09cm);
\filldraw (6,4) circle (0.09cm);
\filldraw (6,5) circle (0.09cm);
\filldraw (6,6) circle (0.09cm);
\filldraw (6,7) circle (0.09cm);
\filldraw (6,8) circle (0.09cm);
\filldraw (6,9) circle (0.09cm);
\filldraw (7,1) circle (0.09cm);
\filldraw (7,2) circle (0.09cm);
\filldraw (7,3) circle (0.09cm);
\filldraw (7,4) circle (0.09cm);
\filldraw (7,5) circle (0.09cm);
\filldraw (7,6) circle (0.09cm);
\filldraw (7,7) circle (0.09cm);
\filldraw (7,8) circle (0.09cm);
\filldraw (7,9) circle (0.09cm);
\filldraw (8,1) circle (0.09cm);
\filldraw (8,2) circle (0.09cm);
\filldraw (8,3) circle (0.09cm);
\filldraw (8,4) circle (0.09cm);
\filldraw (8,5) circle (0.09cm);
\filldraw (8,6) circle (0.09cm);
\filldraw (8,7) circle (0.09cm);
\filldraw (8,8) circle (0.09cm);
\filldraw (8,9) circle (0.09cm);
\filldraw (9,1) circle (0.09cm);
\filldraw (9,2) circle (0.09cm);
\filldraw (9,3) circle (0.09cm);
\filldraw (9,4) circle (0.09cm);
\filldraw (9,5) circle (0.09cm);
\filldraw (9,6) circle (0.09cm);
\filldraw (9,7) circle (0.09cm);
\filldraw (9,8) circle (0.09cm);
\filldraw (9,9) circle (0.09cm);
\end{tikzpicture}
&
\begin{tikzpicture}[scale=0.4]
\draw [thick, ->]  (0,0) -- (10,0);
\draw [thick, ->]  (0,0) -- (0,10);
\draw[ultra thick, domain=0:6,samples=100,smooth,variable=\x] 
plot ({\x},{sqrt(2)*6-sqrt(2)*\x)});
\filldraw (1,1) circle (0.09cm);
\filldraw (1,2) circle (0.09cm);
\filldraw (1,3) circle (0.09cm);
\filldraw (1,4) circle (0.09cm);
\filldraw (1,5) circle (0.09cm);
\filldraw (1,6) circle (0.09cm);
\filldraw (1,7) circle (0.09cm);
\filldraw (1,8) circle (0.09cm);
\filldraw (1,9) circle (0.09cm);
\filldraw (2,1) circle (0.09cm);
\filldraw (2,2) circle (0.09cm);
\filldraw (2,3) circle (0.09cm);
\filldraw (2,4) circle (0.09cm);
\filldraw (2,5) circle (0.09cm);
\filldraw (2,6) circle (0.09cm);
\filldraw (2,7) circle (0.09cm);
\filldraw (2,8) circle (0.09cm);
\filldraw (2,9) circle (0.09cm);
\filldraw (3,1) circle (0.09cm);
\filldraw (3,2) circle (0.09cm);
\filldraw (3,3) circle (0.09cm);
\filldraw (3,4) circle (0.09cm);
\filldraw (3,5) circle (0.09cm);
\filldraw (3,6) circle (0.09cm);
\filldraw (3,7) circle (0.09cm);
\filldraw (3,8) circle (0.09cm);
\filldraw (3,9) circle (0.09cm);
\filldraw (4,1) circle (0.09cm);
\filldraw (4,2) circle (0.09cm);
\filldraw (4,3) circle (0.09cm);
\filldraw (4,4) circle (0.09cm);
\filldraw (4,5) circle (0.09cm);
\filldraw (4,6) circle (0.09cm);
\filldraw (4,7) circle (0.09cm);
\filldraw (4,8) circle (0.09cm);
\filldraw (4,9) circle (0.09cm);
\filldraw (5,1) circle (0.09cm);
\filldraw (5,2) circle (0.09cm);
\filldraw (5,3) circle (0.09cm);
\filldraw (5,4) circle (0.09cm);
\filldraw (5,5) circle (0.09cm);
\filldraw (5,6) circle (0.09cm);
\filldraw (5,7) circle (0.09cm);
\filldraw (5,8) circle (0.09cm);
\filldraw (5,9) circle (0.09cm);
\filldraw (6,1) circle (0.09cm);
\filldraw (6,2) circle (0.09cm);
\filldraw (6,3) circle (0.09cm);
\filldraw (6,4) circle (0.09cm);
\filldraw (6,5) circle (0.09cm);
\filldraw (6,6) circle (0.09cm);
\filldraw (6,7) circle (0.09cm);
\filldraw (6,8) circle (0.09cm);
\filldraw (6,9) circle (0.09cm);
\filldraw (7,1) circle (0.09cm);
\filldraw (7,2) circle (0.09cm);
\filldraw (7,3) circle (0.09cm);
\filldraw (7,4) circle (0.09cm);
\filldraw (7,5) circle (0.09cm);
\filldraw (7,6) circle (0.09cm);
\filldraw (7,7) circle (0.09cm);
\filldraw (7,8) circle (0.09cm);
\filldraw (7,9) circle (0.09cm);
\filldraw (8,1) circle (0.09cm);
\filldraw (8,2) circle (0.09cm);
\filldraw (8,3) circle (0.09cm);
\filldraw (8,4) circle (0.09cm);
\filldraw (8,5) circle (0.09cm);
\filldraw (8,6) circle (0.09cm);
\filldraw (8,7) circle (0.09cm);
\filldraw (8,8) circle (0.09cm);
\filldraw (8,9) circle (0.09cm);
\filldraw (9,1) circle (0.09cm);
\filldraw (9,2) circle (0.09cm);
\filldraw (9,3) circle (0.09cm);
\filldraw (9,4) circle (0.09cm);
\filldraw (9,5) circle (0.09cm);
\filldraw (9,6) circle (0.09cm);
\filldraw (9,7) circle (0.09cm);
\filldraw (9,8) circle (0.09cm);
\filldraw (9,9) circle (0.09cm);
\end{tikzpicture}

\end{tabular}
\caption{Among triangles with fixed area, which contains the most lattice points?}
\end{figure}
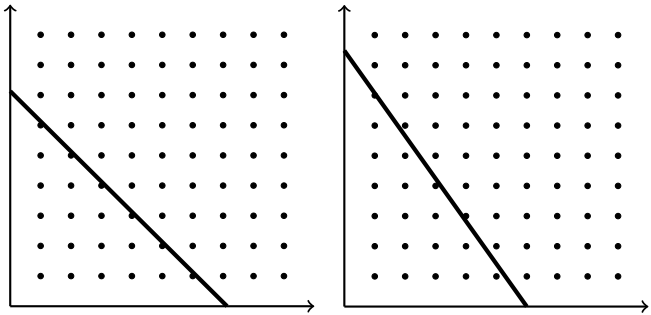

A standard lattice point counting method is to use the Poisson summation formula
to express the number of lattice points in a domain as the area of the domain
plus an error term related to the decay of the Fourier transform of the
indicator function of the domain, which in turn is related to the curvature of
the boundary of the domain. It is therefore not surprising that a lack of
curvature would yield different behavior and require different techniques. Here,
both number theory and dynamical systems start to play a role. The purpose of
this paper is to establish several basic results, and to prove that there is
indeed no limiting shape. In fact, we prove there is an infinite number of
triangles each of which is optimal for an infinite number of arbitrarily large
areas. We emphasize that these questions seem to be remarkably rich, and many
open problems remain (some of which are discussed in \S \ref{open}).

\section{Main Results}
\subsection{Two problems.} There are two ways to approach the problem: pick
a large area and ask which triangle maximizes the number of enclosed lattice
points, or fix a specific triangle and try to understand the number of
lattice points it contains as it is dilated. The question described in the
introduction asks the first question and, a priori, the second problem is strictly
simpler. However, it turns out that it
is possible to obtain sufficiently good control on the error estimates to pursue
the second approach and obtain results for the original question. We start
by presenting some results for the simpler question and then explain how
the techniques can be adapted to deal with the harder problem.

\subsection{Dilating fixed triangles}
Let $N_\beta(\alpha)$ denote the number of positive lattice points contained in
the triangle with vertices $(0,0), (\alpha/\sqrt{\beta},0),$ and
$(0,\alpha \sqrt{\beta})$, i.e.,
\[ 
N_\beta(\alpha) = \# \left\{(k,m) \in \mathbb{Z}_{>0}^2 : m \le \alpha
\sqrt{\beta}  - k \beta \right\},
\]
where $\mathbb{Z}_{>0}$ is the set of positive integers. The slope of the
hypotenuse of this triangle is $-\beta$ so we refer to $\beta$ as the slope
parameter or simply the `slope' of the triangle.  Similarly, since $\alpha^2/2$
is the area of this triangle we refer to $\alpha$
as the area parameter. The subtleties of estimating $N_\beta(\alpha)$ occur near
the boundary of the given triangle. An estimate based on the length of the
boundary shows that 
$$ 
N_\beta(\alpha) = \frac{\alpha^2}{2} + \mathcal{O}_{\beta}(\alpha), \quad
\text{as} \quad \alpha \rightarrow \infty,
$$ 
where the implicit constant in the error term depends on $\beta$. This suggests
$$
\frac{N_{\beta}(\alpha) - \alpha^2/2}{\alpha} \qquad \mbox{as a suitable
renormalization},
$$
which isolates the interesting behavior happening at the linear scale.

\begin{thm} \label{thm1} The limit
$$ 
\lim_{\alpha \rightarrow \infty}{  \frac{N_{\beta}(\alpha) - \alpha^2/2}{\alpha}
} \qquad \mbox{exists if and only if}~\beta~\mbox{is irrational},
$$
and if the limit exists, then it is smaller than $-1$. Moreover, we have that
the set
$$ 
\Lambda \overset{\text{def}}{=} \left\{ \beta \in \mathbb{Q}: \limsup_{\alpha
\rightarrow \infty}{  \frac{N_{\beta}(\alpha) - \alpha^2/2}{\alpha}  } > -1\right\}$$
is non-empty, contained in $[1/3, 3]$, has $1$ as a unique accumulation point,
and has Minkowski dimension at most $3/4$. Moreover, for any finite subset
$\Gamma \subset \Lambda$, there exists $\beta \in \Lambda \setminus \Gamma,$
such that
$$ \limsup_{\alpha \rightarrow \infty}{ \left( N_{\beta}(\alpha) -  \max_{\gamma \in \Gamma}N_{\gamma}(\alpha) \right)} > 0.$$
\end{thm}

The result may be summarized as follows: if one is interested in triangles that,
at least for a sequence of areas tending toward infinity, capture a lot of
lattice points relative to other triangles of equal area, then the slope should
not be irrational: for irrational slopes, we have $N_{\beta}(\alpha) <
\alpha^2/2 - \alpha + o_\beta(\alpha)$. On the other hand, there is an infinite
set $\Lambda$ of rational slopes such that for all $\beta \in \Lambda$ we have
$N_\beta(\alpha) \ge \alpha^2/2 - (1 - \delta_\beta) \alpha + o_\beta(\alpha)$
for infinitely many arbitrarily large $\alpha$ (depending on $\beta$), where
$\delta_\beta >0$ is a positive constant depending on $\beta$.  The set
$\Lambda$ (see Figure 3) has a rather nontrivial structure and is fractal
in the sense that its Minkowski dimension is at most $3/4$. 

\begin{figure}[h!]
\begin{tikzpicture}[scale=0.8]
\draw [thick]  (-1,3) -- (1,-3);
\draw [thick]  (-9,3) -- (9,-3);
\draw [dashed, thick]  (-3,3) -- (3,-3);
\draw [thick]  (-1.5,3) -- (1.5,-3);
\draw [thick]  (-6,3) -- (6,-3);
\draw [thick]  (-2,3) -- (2,-3);
\draw [thick]  (-4.5,3) -- (4.5,-3);
\draw [thick]  (-4,3) -- (4,-3);
\draw [thick]  (-4,3) -- (4,-3);
\draw [thick]  (-12/5,3) -- (12/5,-3);
\draw [thick]  (-15/4,3) -- (15/4,-3);
\draw [thick]  (-9/5,3) -- (9/5,-3);
\draw [thick]  (-5,3) -- (5,-3);
\draw [thick]  (-15/6,3) -- (15/6,-3);
\draw [thick]  (-18/5,3) -- (18/5,-3);
\draw [thick]  (-18/7,3) -- (18/7,-3);
\draw [thick]  (-21/6,3) -- (21/6,-3);
\draw [thick]  (-21/8,3) -- (21/8,-3);
\draw [thick]  (-24/7,3) -- (24/7,-3);
\draw [thick]  (-24/9,3) -- (24/9,-3);
\draw [ultra thick]  (-24/7,3) -- (24/7,-3);
\draw [ultra thick]  (-24/9,3) -- (24/9,-3);
\draw [ultra thick]  (-3.4,3) -- (3.4,-3);
\draw [ultra thick]  (-3.3,3) -- (3.3,-3);
\draw [ultra thick]  (-3.2,3) -- (3.2,-3);
\draw [ultra thick]  (-2.6,3) -- (2.6,-3);
\draw [ultra thick]  (-2.7,3) -- (2.7,-3);
\draw [ultra thick]  (-2.8,3) -- (2.8,-3);
\end{tikzpicture}
\caption{Slopes in $\Lambda$ start to cluster (not shown) around the slope $-1$ (dashed).} 
\end{figure}
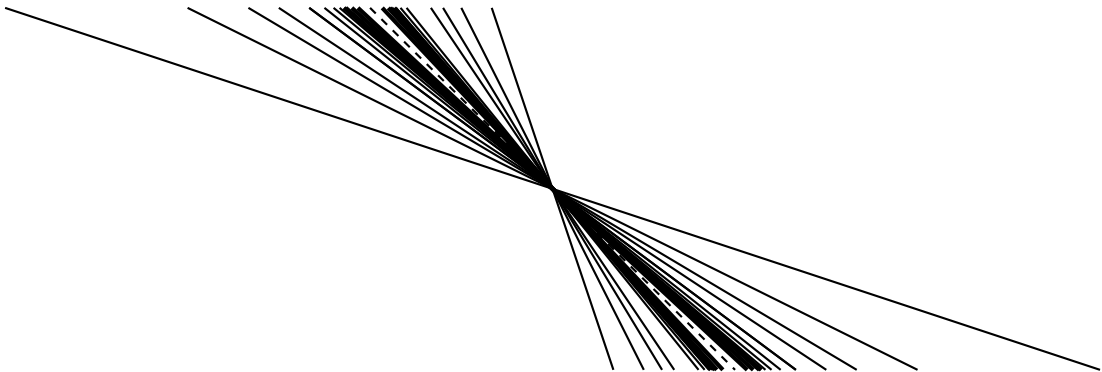

Finally, the dynamics of $\Lambda$ in terms of capturing lattice points are
nontrivial: every finite subset $\Gamma \subset \Lambda$ is at least sometimes
uniformly worse at capturing lattice points than some element $\beta \in \Lambda
\setminus \Gamma$.

\subsection{Slopes that are optimal for arbitrarily large areas}
We define the limit set $S$ as the set of slopes that capture a maximal
number of lattice points for arbitrarily large areas
$$
S = \bigcap_{r >0 } \bigcup_{\alpha > r}{ \argmax_{\beta >0}
N_\beta(\alpha)} .
$$
The next theorem confirms the suspicion of Laugesen \& Liu by establishing that
the limit set is nontrivial. More precisely, we show that the limit set contains
infinitely many elements of $\Lambda$.

\begin{thm} \label{limitset}
There is an infinite subset of $\Lambda$ contained in $S$:
$$
\# \{ p/q \in \Lambda \cap S \} = \infty.
$$
\end{thm}
While we do not have a precise description of the infinite subset of $\Lambda$
which is contained in $S$, the proof of Theorem \ref{limitset} implies that for
every squarefree number $k \in \mathbb{N}$ (i.e. a number that does not contain
any prime factor more than once) there exists an element $p/q \in \Lambda \cap
S$ such that $pq = k s^2$ where $s \in \mathbb{N}$.

\subsection{Bad areas exist} 
In Theorem \ref{thm1} we established that if $\beta$ is irrational, then the
limit
$$ 
\lim_{\alpha \rightarrow \infty}{  \frac{N_{\beta}(\alpha) - \alpha^2/2}{\alpha}
} \quad \text{exists and is strictly less than $-1 $}.
$$
In fact, as we will see in Lemma \ref{lem:irrational}, it is possible to choose
an irrational slope $\beta$ such that the above limit is arbitrarily close to
$-1$. Informally speaking, this means that there exists a fixed triangle with an
irrational slope which captures at least $\alpha^2/2 - (1+\delta) \alpha$
lattice points for all sufficiently large areas, where $\delta >0$ is an
arbitrary fixed constant.  The following Theorem states that there exists `bad'
areas where it is difficult to do much better than this.

\begin{thm}[Bad areas exist] \label{badarea} We have
$$ 
\liminf_{\alpha \rightarrow \infty}{ \sup_{\beta} \frac{N_\beta (\alpha) - \alpha^2/2}{\alpha}} = -1.
$$
\end{thm}
\begin{figure}[h!]
\centering
\includegraphics[width=.6\textwidth]{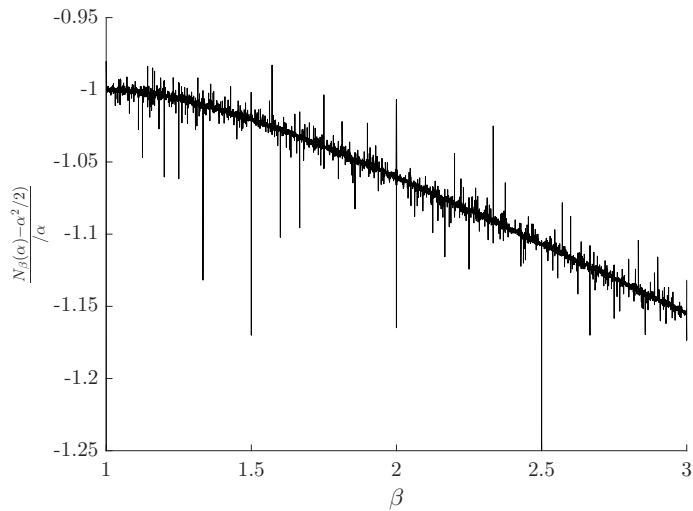} 
\caption{A `bad' area at $\alpha = 15541.957707$. No triangle with area
$\alpha^2/2$ is particularly good at capturing lattice points. The downward sloping
curve is $(-\sqrt{\beta}-\sqrt{1/\beta})/2$ (see
Lemma \ref{lem:irrational}). } \label{fig:badarea}
\end{figure}

We illustrate Theorem \ref{badarea} in Figure \ref{fig:badarea}: a `bad' area at
$\alpha = 15541.957707$ was found numerically (essentially by aligning many
rational slopes to perform poorly via Lemmas \ref{mimic} and \ref{align}), where 
$$ 
\sup_{ \beta } \frac{N_\beta (\alpha) - \alpha^2/2}{\alpha} \approx -0.98035,
$$ 
which is close to the worse case value of $-1$.

\subsection{Rational and irrational slopes.} Along
the way to the proof of Theorem \ref{thm1} we obtain several smaller
results; in particular, we obtain fairly explicit control of the behavior of
$N_{\beta}(\alpha)$ for rational slopes $\beta \in \mathbb{Q}$.

\begin{lemma}[Rational slopes] \label{lem:rational}
Suppose that $p$ and $q$ are positive coprime integers. Then
\[
N_{p/q}(\alpha) = \frac{\alpha^2}{2} - \frac{\sqrt{\frac{p}{q}} + \sqrt{\frac{q}{p}} -
 \sqrt{\frac{1}{pq}} \left(1-2\{\alpha \sqrt{p q}\} \right)
 }{2} \alpha + \mathcal{O}_{p,q}(1), \quad \text{as} \quad \alpha \rightarrow
\infty,
\]
where $\{x \} := x - \lfloor x \rfloor$ denotes the fractional part of $x$.
\end{lemma}

Observe that the coefficient of $\alpha$ in this asymptotic formula is
oscillatory. For example, if $p=q=1$, then Lemma \ref{lem:rational} implies that
$$ 
N_{1}(\alpha) = \frac{\alpha^2}{2} - \left(\frac{1}{2} + \{\alpha\} \right)
\alpha + \mathcal{O}(1),
$$ 
which may be understood as periodic oscillation around $\alpha^2/2 - \alpha$. In
contrast, we have the following asymptotic result for irrational slopes. 

\begin{lemma}[Irrational slopes] \label{lem:irrational}
Suppose that $\beta > 0$ is irrational. Then 
$$
N_{\beta}(\alpha) = \frac{\alpha^2}{2} - \frac{\sqrt{\beta} +
\sqrt{1/\beta}  }{2} \alpha + o_\beta(\alpha), \quad \text{as} \quad \alpha \rightarrow
\infty.
$$
\end{lemma}
Here, the coefficient of $\alpha$ is non-oscillatory, and thus, the behavior of
$N_\beta(\alpha)$ eventually stabilizes relative to $\alpha$. However, by the
arithmetic mean geometric mean inequality
$$
\frac{\sqrt{\beta} +
\sqrt{1/\beta}  }{2}  \ge 1  \quad \text{with equality if and only if $\beta =
1$},
$$
and therefore, if $\beta > 0$ is irrational, then there exists a constant
$\delta_\beta > 0$ such that
$$
N_\beta(\alpha) < \frac{\alpha^2}{2} - (1+\delta_\beta) \alpha
$$
for sufficiently large $\alpha$ (depending on $\beta$). Thus, when $\alpha$ is
sufficiently large (depending on $\beta$) the number of lattice points captured
by a triangle with an irrational slope $\beta$ is always less than the number
captured by some triangle with a rational slope. Indeed, consider the rational
slope $(n+1)/n$. A Taylor expansion of the result of Lemma \ref{lem:rational}
for this slope yields
$$ 
N_{(n+1)/n}(\alpha) \geq \frac{\alpha^2}{2} - \left(1 +
\frac{1}{2n}\right)\alpha + \mathcal{O}_{n}(1).
$$
Hence, if $1/(2n) < \delta_\beta$, then $ N_{(n+1)/n}(\alpha) > N_\beta(\alpha)
$ for all sufficiently large $\alpha$ depending on $\beta$ and $n$ (which can be
chosen in terms of $\beta$).

\subsection{Optimality of the right isosceles triangle}

The original question of Laugesen \& Liu was based upon the conjecture that
triangles which capture a maximal number of lattice points may not approach the
right isosceles triangle in the large area limit. However, intuitively the right
isosceles triangle should perform quite well when its hypotenuse intersects
lattice points.  We show that if $\alpha = n$, where $n \in \mathbb{N}$, then
the isosceles triangle captures strictly more elements than any other triangle.
However, at the same time, the right isosceles triangle is only better than a
generic irrational slope close to 1 for slightly more than half the time
(results of this type were suspected in Laugesen \& Liu \cite{LaugesenLiu2016},
see their \S 9).

\begin{proposition} \label{prop1} 
We have, for every $n \in \mathbb{N}$,
$$ N_{1}\left(n\right) = \frac{n(n-1)}{2}  > \sup_{\beta \neq 1}{ N_{\beta}\left(n\right)}.
$$
If $\beta \in \mathbb{R} \setminus \mathbb{Q}$ is a positive irrational number, then
$$ \lim_{T \rightarrow \infty}{ \frac{1}{T} \left| \left\{ 0 \leq \alpha \leq T:
N_{1}(\alpha) < N_{\beta}(\alpha)\right\} \right| }= \begin{cases}
\frac{3 - \sqrt{\beta} - \sqrt{1/\beta}}{2} \qquad &\mbox{if}~\beta \in \left( \frac{7-3\sqrt{5}}{2}, \frac{7+3\sqrt{5}}{2}\right) \\
0 \qquad &\mbox{otherwise.}
\end{cases}
$$
\end{proposition}

\subsection{Open problems.} \label{open} Many open problems remain. We only list
the few that naturally arise out of the results in this paper; it does seem like
a particularly fruitful area of research.

\begin{enumerate}
\item Can the limit set be completely characterized? Specifically, is there
an explicit subset $\Gamma \subseteq \Lambda$ such that $ \Gamma = S?  $
\item Suppose we define the extended limit set
$$
\tilde{S}  =  \bigcap_{r > 0} \overline{\bigcup_{\alpha > r} \argmax_{\beta > 0}
N_\beta(\alpha)}.
$$
Clearly $S \subseteq \tilde{S}$, and thus $\tilde{S}$ also contains an infinite
subset of $\Lambda$, but does $S = \tilde{S}$?
\item Can any of these results be extended to polygonal shapes? What happens
for shapes that are curved but contain a straight line segment somewhere?
\item The intuition coming from Fourier analysis suggests that a convex curve
with vanishing curvature at a point should still be somewhat well behaved -- is
it possible to get precise results in this intermediate case between strictly
convex and flat line segments?
\item It is well understood that the natural analogue of Pick's theorem, a crucial ingredient in our approach, does not hold in higher dimensions. A substitute is given by the notion of Ehrhart polynomials (see e.g. the very nice book of Beck \& Robins \cite{beck}).
Is it possible to adapt our approach to attack higher-dimensional problems 
by replacing our use of Pick's formula (which we only use in a very mild way) by Ehrhart polynomials?
\end{enumerate}

\subsection{Organization and Notation.} We say
$$
f(x) \lesssim_h g(x) \quad \text{if and only if} \quad f(x) \le C_h \cdot g(x),
$$
for a fixed constant $C_h > 0$ that only depends on $h$. Similarly, we say $f(x)
\gtrsim_h g(x)$  if and only if $g(x) \lesssim_h f(x)$, and we say $f(x) \sim_h
g(x)$ when $f(x) \lesssim_h g(x)$ and $f(x) \gtrsim_h g(x)$. If the implicit
constant does not depend on a parameter $h$, then we simply write $\lesssim$,
$\gtrsim$, and $\sim$, respectively.  The remainder of this paper is organized as
follows. \S \ref{sec:tools} discusses Pick's theorem, establishes basic results
for irrational slopes, and presents a basic fact from Number Theory. \S
\ref{sec:proofthm1} is devoted to the proof of Theorem 1. \S \ref{sec:lemmas}
establishes basic results that will be required for the proofs of Theorem
\ref{limitset} and Theorem \ref{badarea}.  Finally, Theorems \ref{limitset} and
\ref{badarea} are proven in \S \ref{sec:proofthm23}.

\section{Some Useful Tools} \label{sec:tools}
\subsection{Pick's Theorem}
Pick's Theorem \cite{pick} is a classical statement relating the
area $A$ of a polygon whose vertices are on integer lattice points to the number
$I$ of interior lattice points and the number $B$ of lattice points on the
boundary via
\[
A = I + \frac{B}{2} - 1.
\]

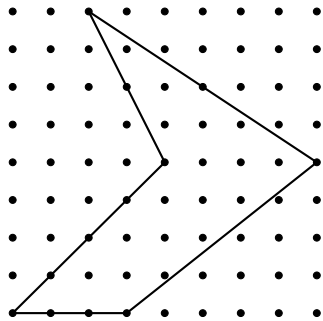
\begin{figure}[h!]
\begin{tikzpicture}[scale=0.5]
\draw [thick]  (1,1) -- (5,5);
\draw [thick]  (5,5) -- (3,9);
\draw [thick]  (3,9) -- (9,5);
\draw [thick]  (9,5) -- (4,1);
\draw [thick]  (4,1) -- (1,1);
\filldraw (1,1) circle (0.09cm);
\filldraw (1,2) circle (0.09cm);
\filldraw (1,3) circle (0.09cm);
\filldraw (1,4) circle (0.09cm);
\filldraw (1,5) circle (0.09cm);
\filldraw (1,6) circle (0.09cm);
\filldraw (1,7) circle (0.09cm);
\filldraw (1,8) circle (0.09cm);
\filldraw (1,9) circle (0.09cm);
\filldraw (2,1) circle (0.09cm);
\filldraw (2,2) circle (0.09cm);
\filldraw (2,3) circle (0.09cm);
\filldraw (2,4) circle (0.09cm);
\filldraw (2,5) circle (0.09cm);
\filldraw (2,6) circle (0.09cm);
\filldraw (2,7) circle (0.09cm);
\filldraw (2,8) circle (0.09cm);
\filldraw (2,9) circle (0.09cm);
\filldraw (3,1) circle (0.09cm);
\filldraw (3,2) circle (0.09cm);
\filldraw (3,3) circle (0.09cm);
\filldraw (3,4) circle (0.09cm);
\filldraw (3,5) circle (0.09cm);
\filldraw (3,6) circle (0.09cm);
\filldraw (3,7) circle (0.09cm);
\filldraw (3,8) circle (0.09cm);
\filldraw (3,9) circle (0.09cm);
\filldraw (4,1) circle (0.09cm);
\filldraw (4,2) circle (0.09cm);
\filldraw (4,3) circle (0.09cm);
\filldraw (4,4) circle (0.09cm);
\filldraw (4,5) circle (0.09cm);
\filldraw (4,6) circle (0.09cm);
\filldraw (4,7) circle (0.09cm);
\filldraw (4,8) circle (0.09cm);
\filldraw (4,9) circle (0.09cm);
\filldraw (5,1) circle (0.09cm);
\filldraw (5,2) circle (0.09cm);
\filldraw (5,3) circle (0.09cm);
\filldraw (5,4) circle (0.09cm);
\filldraw (5,5) circle (0.09cm);
\filldraw (5,6) circle (0.09cm);
\filldraw (5,7) circle (0.09cm);
\filldraw (5,8) circle (0.09cm);
\filldraw (5,9) circle (0.09cm);
\filldraw (6,1) circle (0.09cm);
\filldraw (6,2) circle (0.09cm);
\filldraw (6,3) circle (0.09cm);
\filldraw (6,4) circle (0.09cm);
\filldraw (6,5) circle (0.09cm);
\filldraw (6,6) circle (0.09cm);
\filldraw (6,7) circle (0.09cm);
\filldraw (6,8) circle (0.09cm);
\filldraw (6,9) circle (0.09cm);
\filldraw (7,1) circle (0.09cm);
\filldraw (7,2) circle (0.09cm);
\filldraw (7,3) circle (0.09cm);
\filldraw (7,4) circle (0.09cm);
\filldraw (7,5) circle (0.09cm);
\filldraw (7,6) circle (0.09cm);
\filldraw (7,7) circle (0.09cm);
\filldraw (7,8) circle (0.09cm);
\filldraw (7,9) circle (0.09cm);
\filldraw (8,1) circle (0.09cm);
\filldraw (8,2) circle (0.09cm);
\filldraw (8,3) circle (0.09cm);
\filldraw (8,4) circle (0.09cm);
\filldraw (8,5) circle (0.09cm);
\filldraw (8,6) circle (0.09cm);
\filldraw (8,7) circle (0.09cm);
\filldraw (8,8) circle (0.09cm);
\filldraw (8,9) circle (0.09cm);
\filldraw (9,1) circle (0.09cm);
\filldraw (9,2) circle (0.09cm);
\filldraw (9,3) circle (0.09cm);
\filldraw (9,4) circle (0.09cm);
\filldraw (9,5) circle (0.09cm);
\filldraw (9,6) circle (0.09cm);
\filldraw (9,7) circle (0.09cm);
\filldraw (9,8) circle (0.09cm);
\filldraw (9,9) circle (0.09cm);
\end{tikzpicture}
\caption{A polygon with 17 interior lattice points, $12$ boundary lattice points, and area $22$.} 
\end{figure}

The triangles that we are interested in do not, in general, have vertices on
lattice points, and therefore, Pick's Theorem does not directly apply.
Rather, given a triangle $T$ we will consider the convex hull $C$ of the
lattice points contained in $T$. The convex hull $C$ is a polygonal domain whose
vertices are on lattice points, which contains the same number of lattice points
as $T$. By Pick's Theorem, the total number $I+B$ of lattice points contained in $C$ is
$$
I+B = A + \frac{B}{2} + 1,
$$
where $I$ is the number of interior lattice points in $C$, $B$ is the number of
lattice points on the boundary of $C$, and $A$ is the area of $C$. Thus, by Pick's
Theorem, we have reduced the problem of determining the number of lattice points
in $T$ to estimating the area of the convex hull $C$ as well as the number of
lattice points on the boundary of $C$.

\subsection{Irrational slopes.}
This section presents a self-contained geometric-combinatorial characterization
of irrationality. Let $C_{\beta,\gamma}(N)$ denote the convex hull of the
nonnegative lattice points under the line $y = \beta x + \gamma$ whose
$x$-coordinate is at most $N$. More precisely,
$$
C_{\beta,\gamma}(N) := \text{convex hull}\left( \{ (k,m) \in \mathbb{Z}_{\ge
0}^2 : k \le N \wedge m \le \beta k + \gamma \} \right).
$$

\begin{figure}[h!]
\centering
\begin{tikzpicture}[scale=0.6]
\draw [thick, ->]  (0,0) -- (10,0);
\draw [thick, ->]  (0,0) -- (0,10);
\filldraw (0,0) circle (0.09cm);
\filldraw (1,0) circle (0.09cm);
\filldraw (2,0) circle (0.09cm);
\filldraw (3,0) circle (0.09cm);
\filldraw (4,0) circle (0.09cm);
\filldraw (5,0) circle (0.09cm);
\filldraw (6,0) circle (0.09cm);
\filldraw (7,0) circle (0.09cm);
\filldraw (8,0) circle (0.09cm);
\filldraw (9,0) circle (0.09cm);
\filldraw (0,1) circle (0.09cm);
\filldraw (0,2) circle (0.09cm);
\filldraw (0,3) circle (0.09cm);
\filldraw (0,4) circle (0.09cm);
\filldraw (0,5) circle (0.09cm);
\filldraw (0,6) circle (0.09cm);
\filldraw (0,7) circle (0.09cm);
\filldraw (0,8) circle (0.09cm);
\filldraw (0,9) circle (0.09cm);
\filldraw (1,1) circle (0.09cm);
\filldraw (1,2) circle (0.09cm);
\filldraw (1,3) circle (0.09cm);
\filldraw (1,4) circle (0.09cm);
\filldraw (1,5) circle (0.09cm);
\filldraw (1,6) circle (0.09cm);
\filldraw (1,7) circle (0.09cm);
\filldraw (1,8) circle (0.09cm);
\filldraw (1,9) circle (0.09cm);
\filldraw (2,1) circle (0.09cm);
\filldraw (2,2) circle (0.09cm);
\filldraw (2,3) circle (0.09cm);
\filldraw (2,4) circle (0.09cm);
\filldraw (2,5) circle (0.09cm);
\filldraw (2,6) circle (0.09cm);
\filldraw (2,7) circle (0.09cm);
\filldraw (2,8) circle (0.09cm);
\filldraw (2,9) circle (0.09cm);
\filldraw (3,1) circle (0.09cm);
\filldraw (3,2) circle (0.09cm);
\filldraw (3,3) circle (0.09cm);
\filldraw (3,4) circle (0.09cm);
\filldraw (3,5) circle (0.09cm);
\filldraw (3,6) circle (0.09cm);
\filldraw (3,7) circle (0.09cm);
\filldraw (3,8) circle (0.09cm);
\filldraw (3,9) circle (0.09cm);
\filldraw (4,1) circle (0.09cm);
\filldraw (4,2) circle (0.09cm);
\filldraw (4,3) circle (0.09cm);
\filldraw (4,4) circle (0.09cm);
\filldraw (4,5) circle (0.09cm);
\filldraw (4,6) circle (0.09cm);
\filldraw (4,7) circle (0.09cm);
\filldraw (4,8) circle (0.09cm);
\filldraw (4,9) circle (0.09cm);
\filldraw (5,1) circle (0.09cm);
\filldraw (5,2) circle (0.09cm);
\filldraw (5,3) circle (0.09cm);
\filldraw (5,4) circle (0.09cm);
\filldraw (5,5) circle (0.09cm);
\filldraw (5,6) circle (0.09cm);
\filldraw (5,7) circle (0.09cm);
\filldraw (5,8) circle (0.09cm);
\filldraw (5,9) circle (0.09cm);
\filldraw (6,1) circle (0.09cm);
\filldraw (6,2) circle (0.09cm);
\filldraw (6,3) circle (0.09cm);
\filldraw (6,4) circle (0.09cm);
\filldraw (6,5) circle (0.09cm);
\filldraw (6,6) circle (0.09cm);
\filldraw (6,7) circle (0.09cm);
\filldraw (6,8) circle (0.09cm);
\filldraw (6,9) circle (0.09cm);
\filldraw (7,1) circle (0.09cm);
\filldraw (7,2) circle (0.09cm);
\filldraw (7,3) circle (0.09cm);
\filldraw (7,4) circle (0.09cm);
\filldraw (7,5) circle (0.09cm);
\filldraw (7,6) circle (0.09cm);
\filldraw (7,7) circle (0.09cm);
\filldraw (7,8) circle (0.09cm);
\filldraw (7,9) circle (0.09cm);
\filldraw (8,1) circle (0.09cm);
\filldraw (8,2) circle (0.09cm);
\filldraw (8,3) circle (0.09cm);
\filldraw (8,4) circle (0.09cm);
\filldraw (8,5) circle (0.09cm);
\filldraw (8,6) circle (0.09cm);
\filldraw (8,7) circle (0.09cm);
\filldraw (8,8) circle (0.09cm);
\filldraw (8,9) circle (0.09cm);
\node[below] at (0,0) {$0$};
\node[below] at (7,0) {$N$};
\draw [thick,dashed] (0,0.3) -- (7, 9.6);
\draw [ultra thick] (0,0) -- (7,0) -- (7,9) -- (6,8) -- (3,4) -- (0,0);
\draw [ultra thick] (7,9) -- (6,8) -- (3,4) -- (0,0);
\end{tikzpicture}
\caption{The line $y = \beta x + \gamma$ (dashed) and the boundary of the convex hull
$C_{\beta,\gamma}(N)$  (bold).}
\label{fig:irr}
\end{figure}
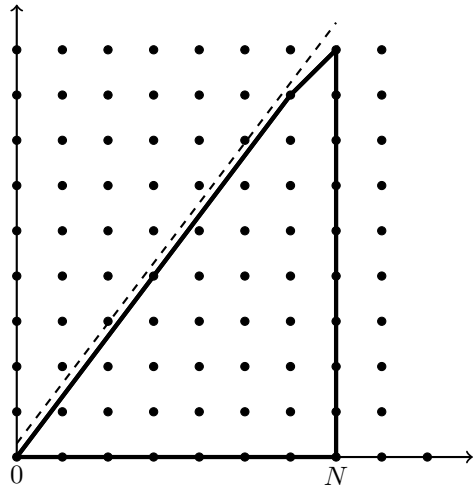

In the following Lemma, we show that $\beta > 0$ is irrational if and only if
the boundary $\partial C_{\beta,\gamma}(N)$ of the convex hull
$C_{\beta,\gamma}(N)$ contains $(1+\beta)N + o_\beta(N)$ lattice points. That is
to say, the part of the boundary of the convex hull that is neither on the
$x$-axis nor on the line $x = N$ contains less than linear lattice points in
$N$.

\begin{lemma}  \label{lem:irrationalhull}
Suppose $\beta > 0$ and $0 \le \gamma < 1$ are arbitrary. Then $\beta > 0$ is
irrational if and only if 
$$
\# \left( \partial C_{\beta,\gamma}(N) \cap \mathbb{Z}^2 \right) = (1+\beta)N +
o_\beta(N), \quad \text{as} \quad N \rightarrow \infty.  
$$
\end{lemma}

\begin{proof}[Proof]
Fix $\beta > 0$ and $0 \le \gamma < 1$. Observe that the number of lattice
points on $\partial C_{\beta,\gamma}(N)$ that are either on the $x$-axis or the
line $x = N$ is
$$
\# \left\{ (k,m) \in  \partial C_{\beta,\gamma}(N) \cap \mathbb{Z}^2 : m = 0
\vee k = N \right\} = (1+\beta)N + \mathcal{O}(1).
$$
Therefore, it suffices to show that the number of lattice points on $\partial
C_{\beta,\gamma}(N)$ that are neither on the $x$-axis nor on the line $x = N$
is
$$
\# \left\{ (k,m) \in  \partial C_{\beta,\gamma}(N) \cap \mathbb{Z}^2 : m > 0
\wedge k < N \right\} = o_\beta(N).
$$
Each of these lattice points has a unique $x$-coordinate so if we define
$$
A := \left\{ k < N : \left( \exists m > 0 : (k,m) \in \partial
C_{\beta,\gamma}(N) \cap \mathbb{Z}^2 \right) \right\},
$$
then it suffices to show that
$$
\# A = o_\beta(N), \quad \text{as} \quad N \rightarrow \infty.
$$
If $\beta$ is rational, then it is not hard to show that $A$ will contain linear
lattice points in $N$ (e.g., see the proof of Lemma \ref{lem:rational}).
Suppose $\beta > 0$ is irrational; we prove by contradiction. Without loss of
generality we may suppose that $\beta \ge 1$ (otherwise, we may consider the
triangle with slope $1/\beta$ which encloses the same number of positive lattice
points). Suppose there exists a constant $\varepsilon > 0$ such that for
arbitrarily large $N$ 
$$ 
\# A \ge \varepsilon N.
$$
The following argument is independent of $\gamma$. We argue as follows: for at
least $\varepsilon N/2$ elements in $A$, it is true that the next element in $A$
is at distance less than $4/\varepsilon$. If that were false, then
$$N \ge \max_{a \in A}{a} - \min_{a \in A}{a} \geq \left(\frac{\varepsilon}{2} N\right) \frac{4}{\varepsilon} \geq 2N,$$
which is a contradiction. We now study the slope of the boundary of the convex
hull $C_{\beta,\gamma}(N)$ between each of these $\varepsilon N/2$ points, and their
following points in $A$. Since each of the following points is at most distance
$4/\varepsilon$ away, it is clear that each slope is a rational number
$p/q$ with denominator less than $4/\varepsilon$ and $p/q \leq \beta + 1$.
The cardinality of this set of slopes is bounded
$$ \# \left\{ \frac{p}{q}: 1 \leq q \leq \frac{4}{\varepsilon} \wedge 1 \leq p
\leq (\beta+1)q \right\} \leq \sum_{q=1}^{\left\lfloor 4/\varepsilon \right\rfloor} (\beta+1)q \le \frac{16 (1+\beta)}{\varepsilon^2}.$$
The second ingredient is a consequence of convexity: consecutive slopes are monotonically decreasing. The third ingredient
is that the slope cannot be constant over too long a stretch: more precisely, let
$$ \delta_{\beta,\varepsilon} = \min\left\{ \left|\beta- \frac{p}{q}\right|: 1 \leq q \leq
\frac{4}{\varepsilon} \wedge 1 \leq p \leq (\beta+1)q  \right\}.
$$
We emphasize that $\delta_{\beta,\varepsilon}$ only depends on $\beta$ and
$\varepsilon$.  Since $\beta$ is irrational, we have that
$\delta_{\beta,\varepsilon} > 0$. Let us now assume that the slope $p/q$ occurs
over a long stretch. If $p/q > \beta$, then $p/q \geq \beta+
\delta_{\beta,\varepsilon}$ and we see that the stretch can be
at most of length $\delta_{\beta,\varepsilon}^{-1}$ (because the line would otherwise intersect
the irrational line). If $p/q < \beta$,
then the line would eventually (depending on $\delta_{\beta,\varepsilon}$) be at distance bigger than 1 from the irrational line and this would allow
us to identify a lattice point outside the convex hull, which would be a contradiction. Altogether, this implies that
$$ \frac{\varepsilon N}{2} \leq  \frac{16 (1+\beta)}{\varepsilon^2}
\frac{1}{\delta_{\beta,\varepsilon}},
$$
and hence
$$
N \leq  \frac{32 (1+\beta)}{\varepsilon^3}
\frac{1}{\delta_{\beta,\varepsilon}} < \infty,
$$ 
which is the desired contradiction.
\end{proof}

We remark that the asymptotic error $o_\beta(N)$ cannot, in general, be
improved because for any fixed $N$ one can take an irrational number
sufficiently close to 1 such that the error term is actually arbitrarily close
to order $N$.  Moreover, the convergence of the error term $o(N)/N$ to 0 can
seen to be arbitrarily slow by considering slopes given by Liouville-type
numbers 
$$ 
\sum_{n=1}^{\infty}{\frac{1}{10^{n!}}}, \quad
\sum_{n=1}^{\infty}{\frac{1}{10^{(n!^{n!})}}}, \dots \qquad \mbox{that are
extremely well approximated by rationals.}
$$
However, the proof can be made quantitative under an additional assumption on
$\beta$ as the next Corollary shows. We have no reason to assume that the
following result is sharp; it seems likely that using more powerful techniques
(the continued fraction expansion of $\beta$) one should be able to obtain much
stronger results. 

\begin{corollary} Let $\mu > 0$. If $\beta$ satisfies the diophantine condition
$$ \forall~\frac{p}{q} \in \mathbb{Q}: \qquad \left|\beta- \frac{p}{q}\right| \geq \frac{c}{q^{\mu}},$$
then
$$ 
\# \left\{ (k,m) \in \partial C_{\beta,\gamma}(N) \cap \mathbb{Z}^2 : m > 0
\wedge k < N \right\} = \mathcal{O}\left(N^{\frac{\mu + 2}{\mu+3}}\right).
$$
\end{corollary}
\begin{proof} We argue as above and observe again that the number of possible small fractions satisfies
$$ \# \left\{ \frac{p}{q}: 1 \leq q \leq \frac{4}{\varepsilon} \wedge 1 \leq p
\leq (\beta+1)q \right\} \leq \frac{16 (1+\beta)}{\varepsilon^2}.$$
Since $q \leq 4 \varepsilon^{-1}$, we have, by assumption,
$$ \delta_{\beta,\varepsilon} = \min\left\{ \left|\beta- \frac{p}{q}\right|: 1
\leq q \leq \frac{4}{\varepsilon} \wedge 1 \leq p \leq (\beta+1)q  \right\} \geq
\frac{c}{(4/\varepsilon)^{\mu}}.$$
This shows that
$$N \le \frac{32 (1+\beta)}{\varepsilon^2} \frac{1}{\delta_{\beta,\varepsilon}}
\le \frac{c^\prime}{\varepsilon^{3+\mu}} \qquad \mbox{and thus} \qquad
\varepsilon \leq \frac{c''}{N^{\frac{1}{3 + \mu}}}.$$ This shows that the
maximum size of the set is $$ \varepsilon N \leq c'' N^{\frac{\mu +
2}{\mu+3}}.$$
\end{proof}

We also need that the area of the convex hull $C_{\beta,\gamma}(N)$ approaches
the area enclosed by the line $y = \beta x + \gamma$, the $x$-axis, and the line
$x = N$, see Figure \ref{fig:irr}.

\begin{lemma}  \label{lem:irrationalarea}
Let $\beta > 0$ be an irrational number, and $0 \leq \gamma < 1$ be arbitrary.
Then
$$
|C_{\beta,\gamma}(N)| = \frac{1}{2} \beta N^2 + \gamma N + o_\beta(N), \quad
\text{as} \quad N \rightarrow \infty.
$$
\end{lemma}
\begin{proof} Let $\varepsilon > 0$ be arbitrary. 
We will prove the existence of an integer $m_{\beta, \varepsilon} \in \mathbb{N}$ with the
property that for all $k \in \mathbb{N}$
$$\max_{k \leq j \leq k + m_{\beta,\varepsilon}} \left\{ \beta j + \gamma \right\} \geq 1 - \varepsilon,$$
where, as usual, $\left\{x\right\} = x - \left\lfloor x \right\rfloor$ denotes
the fractional part of $x$. This will
then establish the statement as follows: it guarantees that in every consecutive
block of $m_{\beta, \varepsilon}$ lattice points one of them is $\varepsilon$
close to the limiting line $y = \beta x + \gamma$: this immediately shows that
the area of the convex hull is at most $\sim \varepsilon N$ away from the area
enclosed by the line $y = \beta x + \gamma$, the $x$-axis, and the line $x = N$,
which implies the result. For convenience of notation, we consider $\left\{
\beta j + \gamma \right\}_{j=1}^{\infty}$ a sequence on the torus $\mathbb{T}
\cong [0,1]$. The desired statement would follow if we knew that there exists
$m_{\beta, \varepsilon} \in \mathbb{N}$ with the property that for all $k \in
\mathbb{N}$ $$  \left\{  \beta j: k \leq j \leq k + m_{\beta,
\varepsilon}\right\} \qquad \mbox{is
a}~\varepsilon\mbox{-net}~\mbox{on}~\mathbb{T}.$$ Recall that a
$\varepsilon$-net is a set of points such that every element of $\mathbb{T}$ is
at most at distance $\varepsilon$ from one
of the points in this net.
Now we exploit the linear structure of the sequence $\beta(k + j) = \beta k + \beta j$. 
The desired statement is rotation invariant, so it is equivalent to show the existence of a
$m_{\beta, \varepsilon} \in \mathbb{N}$ such that 
$$  \left\{  \beta j: 1 \leq j \leq  m_{\beta, \varepsilon}\right\} \qquad
\mbox{is a}~\varepsilon\mbox{-net}~\mbox{on}~\mathbb{T}.$$ 
This is now implied by the fact that a Kronecker sequence with an irrational $\beta$ is uniformly distributed (first established by Hermann Weyl \cite{weyl}) and that uniformly distributed
sequences have the size of the maximal gap tending to 0 (an easy exercise that can be found, for example, in the book of Kuipers \& Niederreiter \cite{kuipers}).
\end{proof}

\subsection{Aligning multiples.} Given a set $\left\{a_1, \dots, a_n\right\}$ of positive real numbers, we may consider their
multiples $(k a_1)_{k=1}^{\infty}$, $(k a_2)_{k=1}^{\infty}, \dots, (k
a_n)_{k=1}^{\infty}$. The main result from this section is that there exists arbitrarily small intervals that contain an element
from each sequence. The statement is a folklore result and a standard application of the Poincar\'{e} recurrence theorem.

\begin{lemma} \label{basicalign} Suppose $\left\{a_1, a_2, \dots, a_n\right\}
\subset \mathbb{R}_{>0}$. Then for all $\varepsilon > 0$, there exists
$\{b_1,b_2,\ldots,b_n\} \subset \mathbb{Z}_{>0}$ and $m \in \mathbb{N}$ such that
$$
\max_{1 \leq i\leq n} \left|a_ib_i - m\right| \le \varepsilon.
$$
\end{lemma}

\begin{proof} We make use of the Poincar\'{e} recurrence theorem: it states that
given a measure space $(X, \mathcal{A}, \mu)$ and a measure-preserving
transformation $T$ from $X$ to itself almost every point from a given set $A \in \mathcal{A}$
with $\mu(A) > 0$ returns to the set $A$ infinitely many times or, formally,
$$ 
\mu \left( \left\{x \in A: \left( \exists k \in \mathbb{N} :  \left( \forall
\ell > k,
~~T^{\ell}(x) \notin A \right ) \right) \right\}\right) = 0,
$$
for every $A \in \mathcal{A}$. We consider the torus $\mathbb{T}^n = X$ equipped
with the Lebesgue measure $\mu$ and consider the measure-preserving
transformation $$ T(x_1, x_2, \dots, x_n) = \left(x_1 + \frac{1}{a_1}, x_2 +
\frac{1}{a_2}, \dots, x_n + \frac{1}{a_n} \right).$$
We apply the Poincar\'{e} recurrence theorem to the set
$$ A = \left\{ x \in \mathbb{T}^n: ~\sup_{1 \leq i \leq n}{|x_i|} \leq
\frac{\varepsilon}{4 \max(a_1, \dots, a_n)} \right\} \in \mathcal{A}.$$
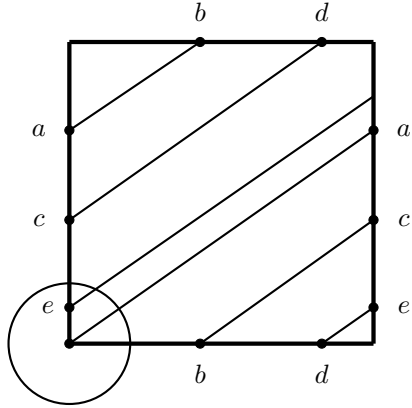
\begin{figure}[h!]
\centering
\begin{tikzpicture}[scale=4]
\draw[ultra thick] (0,0) -- (1,0);
\draw[ultra thick] (1,1) -- (1,0);
\draw[ultra thick] (1,1) -- (0,1);
\draw[ultra thick] (0,1) -- (0,0);
\filldraw (0,0) circle (0.015cm);
\draw [thick] (0,0) circle (0.2cm);
\draw[thick] (0,0) -- (1,1/1.4142);
\draw[thick] (0,1/1.4142) -- (0.43, 1);
\draw[thick] (0.43,0) -- (1, 0.41);
\draw[thick] (0,0.41) -- (0.83, 1);
\draw[thick] (0.83,0) -- (1, 0.12);
\draw[thick] (0, 0.12) -- (1,0.82);
\filldraw (1,1/1.4142) circle (0.015cm);
\node at (1.1,1/1.4142) {$a$};
\filldraw (0,1/1.4142) circle (0.015cm);
\node at (-0.1,1/1.4142) {$a$};
\filldraw (0.43,1) circle (0.015cm);
\node at (0.43, 1.1) {$b$};
\filldraw (0.43,0) circle (0.015cm);
\node at (0.43, -0.1) {$b$};
\filldraw (1,0.41) circle (0.015cm);
\node at (1.1,0.41) {$c$};
\filldraw (0,0.41) circle (0.015cm);
\node at (-0.1,0.41) {$c$};
\filldraw (0.83, 1) circle (0.015cm);
\node at (0.83, 1.1) {$d$};
\filldraw (0.83, 0) circle (0.015cm);
\node at (0.83, -0.1) {$d$};
\filldraw (1,0.12) circle (0.015cm);
\node at (1.1,0.12) {$e$};
\filldraw (0,0.12) circle (0.015cm);
\node at (-0.07,0.12) {$e$};
\end{tikzpicture}
\caption{Every linear flow on the torus eventually returns close to the origin.}
\end{figure}
Since $\mu(A) > 0$, there exists at least one $x_0 \in A$ such that $T(x_0),
T(T(x_0)), \dots$ returns to $A$ infinitely often.  Then, because of the
underlying linearity of $T$, we have that $$ T^{\ell}(0) = T^{\ell}(x_0) - x_0,$$
and thus, with the triangle inequality, for infinitely many $\ell \in
\mathbb{N}$ 
$$ 
T^{\ell}(0) \in \left\{ x \in \mathbb{T}^n: ~\sup_{1 \leq i \leq
n}{|x_i|} \leq \frac{\varepsilon}{2\max(a_1, \dots, a_n)} \right\}.
$$
Put differently, we have for suitable $c_1, c_2, \dots, c_n \in \mathbb{N}$ that
$$ \left| \frac{\ell}{a_i} - c_i\right| \leq  \frac{\varepsilon}{2\max(a_1, \dots, a_n)}.$$
Multiplication with $a_i$ shows that
$$ | \ell - a_i c_i|  \leq \frac{\varepsilon}{2}$$
from which the result follows.
\end{proof}
The argument immediately suggests several ways of how this result could be improved. The worst possible case is when the vector
$(a_1, a_2, \dots, a_n)$ is badly approximable in which case the one-parameter flow $\left\{(a_1 t, a_2 t, \dots, a_n t):  t > 0\right\}$
is effectively exploring the entire Torus and may require a very long time to return to the origin. In contrast, linear dependence,
getting trapped in subspaces or being well approximable by rational numbers,
shorten the return time.

\section{Proof of Theorem \ref{thm1}} \label{sec:proofthm1}
This section is organized as follows.  First, we prove Lemmas \ref{lem:rational}
and \ref{lem:irrational} which provide asymptotic formulas for $N_\beta(\alpha)$
as $\alpha \rightarrow \infty$ for fixed rational and irrational slopes,
respectively. Second, we show that the Minkowski dimension of $\Lambda$ is at
most $3/4$. Third, we use Lemmas \ref{lem:rational} and \ref{lem:irrational} to
complete the proof of Theorem \ref{thm1}.

\subsection{Rational slopes: Lemma \ref{lem:rational}} \label{rationalproof}
If $p$ and $q$ are positive coprime integers, then  Lemma \ref{lem:rational}
states that
\[
N_{p/q}(\alpha) = \frac{\alpha^2}{2} + \frac{-\sqrt{\frac{p}{q}} -
\sqrt{\frac{q}{p}} + \sqrt{\frac{1}{pq}} \left(1-2\{\alpha \sqrt{p q}\} \right)
}{2} \alpha + \mathcal{O}_{p,q}(1).
\]
The interesting behavior of this expression occurs in the coefficient of
$\alpha$. This coefficient is periodic with period $1/\sqrt{pq}$ and
peak-to-peak amplitude $1/\sqrt{pq}$. Furthermore, as $p$ and $q$ tend towards
infinity this coefficient approaches $(-\sqrt{p/q}-\sqrt{q/p})/2$.  This
convergence is suggestive of the behavior of irrational numbers, whose counting
function converges to a quadratic polynomial in $\alpha$ as $\alpha \rightarrow
\infty$.

\begin{proof}[Proof of Lemma \ref{lem:rational}]
Suppose that a triangle with vertices $(0,0)$, $(\alpha \sqrt{q/p},0)$, and
$(0,\alpha \sqrt{p/q})$ is given. If a line of slope $-p/q$ intersects the
$x$-axis at $\alpha \sqrt{q/p}-b$ and the point $(k,m) \in \mathbb{Z}^2$,
then
$$
m = -\frac{p}{q}\left(k - \alpha \sqrt{\frac{q}{p}} + b\right).
$$
Solving for $b$ yields 
$$
b = \frac{\alpha \sqrt{pq} - (qm + kp)}{p}.
$$
The smallest nonnegative value of $b$ that can be written in this form for some
$(k,m) \in \mathbb{Z}^2$ is
$$
b^* = \frac{\alpha \sqrt{pq} - \lfloor \alpha \sqrt{q p} \rfloor}{p}.
$$
Indeed, since $p$ and $q$ are coprime $\{qm + kp: k,m \in \mathbb{Z}\} =
\mathbb{Z}$. Let $T^*$ be the triangle with vertices
$(0,0)$, $(\alpha \sqrt{q/p} -b^*,0)$, and $(0,\alpha \sqrt{p/q} - (p/q)b^*)$,
and $C$ be the convex hull of the nonnegative lattice points enclosed by the
triangle with vertices $(0,0)$, $(\alpha \sqrt{q/p},0)$, and $(0,\alpha
\sqrt{p/q})$.
\begin{figure}[h!]
\centering
\begin{tikzpicture}[scale=.7]
\draw [thick, ->]  (0,0) -- (7.5,0);
\draw [thick, ->]  (0,0) -- (0,9.5);
\draw [ultra thick] (0,0) -- (0,8) -- (1,7) -- (2,6)-- (5,2) -- (6,0) -- (0,0);
\draw [thick] (0,0) -- (0,8.8) -- (6.6,0) -- (0,0);
\draw [thick, dashed] (0,0) -- (0,8.6667) -- (6.5,0) -- (0,0);
\filldraw (0,0) circle (0.09cm);
\filldraw (0,1) circle (0.09cm);
\filldraw (0,2) circle (0.09cm);
\filldraw (0,3) circle (0.09cm);
\filldraw (0,4) circle (0.09cm);
\filldraw (0,5) circle (0.09cm);
\filldraw (0,6) circle (0.09cm);
\filldraw (0,7) circle (0.09cm);
\filldraw (0,8) circle (0.09cm);
\filldraw (0,9) circle (0.09cm);
\filldraw (1,0) circle (0.09cm);
\filldraw (1,1) circle (0.09cm);
\filldraw (1,2) circle (0.09cm);
\filldraw (1,3) circle (0.09cm);
\filldraw (1,4) circle (0.09cm);
\filldraw (1,5) circle (0.09cm);
\filldraw (1,6) circle (0.09cm);
\filldraw (1,7) circle (0.09cm);
\filldraw (1,8) circle (0.09cm);
\filldraw (1,9) circle (0.09cm);
\filldraw (2,0) circle (0.09cm);
\filldraw (2,1) circle (0.09cm);
\filldraw (2,2) circle (0.09cm);
\filldraw (2,3) circle (0.09cm);
\filldraw (2,4) circle (0.09cm);
\filldraw (2,5) circle (0.09cm);
\filldraw (2,6) circle (0.09cm);
\filldraw (2,7) circle (0.09cm);
\filldraw (2,8) circle (0.09cm);
\filldraw (2,9) circle (0.09cm);
\filldraw (3,0) circle (0.09cm);
\filldraw (3,1) circle (0.09cm);
\filldraw (3,2) circle (0.09cm);
\filldraw (3,3) circle (0.09cm);
\filldraw (3,4) circle (0.09cm);
\filldraw (3,5) circle (0.09cm);
\filldraw (3,6) circle (0.09cm);
\filldraw (3,7) circle (0.09cm);
\filldraw (3,8) circle (0.09cm);
\filldraw (3,9) circle (0.09cm);
\filldraw (4,0) circle (0.09cm);
\filldraw (4,1) circle (0.09cm);
\filldraw (4,2) circle (0.09cm);
\filldraw (4,3) circle (0.09cm);
\filldraw (4,4) circle (0.09cm);
\filldraw (4,5) circle (0.09cm);
\filldraw (4,6) circle (0.09cm);
\filldraw (4,7) circle (0.09cm);
\filldraw (4,8) circle (0.09cm);
\filldraw (4,9) circle (0.09cm);
\filldraw (5,0) circle (0.09cm);
\filldraw (5,1) circle (0.09cm);
\filldraw (5,2) circle (0.09cm);
\filldraw (5,3) circle (0.09cm);
\filldraw (5,4) circle (0.09cm);
\filldraw (5,5) circle (0.09cm);
\filldraw (5,6) circle (0.09cm);
\filldraw (5,7) circle (0.09cm);
\filldraw (5,8) circle (0.09cm);
\filldraw (5,9) circle (0.09cm);
\filldraw (6,0) circle (0.09cm);
\filldraw (6,1) circle (0.09cm);
\filldraw (6,2) circle (0.09cm);
\filldraw (6,3) circle (0.09cm);
\filldraw (6,4) circle (0.09cm);
\filldraw (6,5) circle (0.09cm);
\filldraw (6,6) circle (0.09cm);
\filldraw (6,7) circle (0.09cm);
\filldraw (6,8) circle (0.09cm);
\filldraw (6,9) circle (0.09cm);
\filldraw (7,0) circle (0.09cm);
\filldraw (7,1) circle (0.09cm);
\filldraw (7,2) circle (0.09cm);
\filldraw (7,3) circle (0.09cm);
\filldraw (7,4) circle (0.09cm);
\filldraw (7,5) circle (0.09cm);
\filldraw (7,6) circle (0.09cm);
\filldraw (7,7) circle (0.09cm);
\filldraw (7,8) circle (0.09cm);
\filldraw (7,9) circle (0.09cm);
\end{tikzpicture}
\caption{Triangle $T^*$ (dashed) and boundary of convex hull $C$ (bold) for
a given triangle (solid).} \label{fig8}
\end{figure}
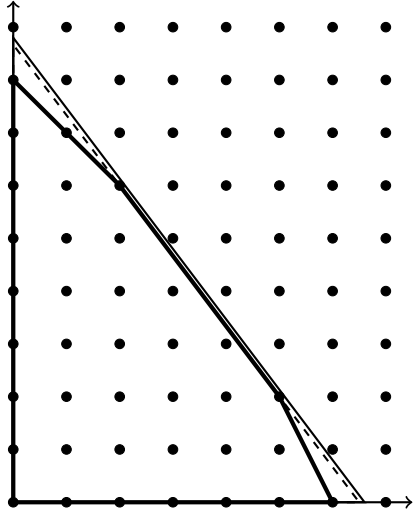

To summarize, $b^* \ge 0$ is the smallest nonnegative number such that the line
$$
y = -\frac{p}{q} \left(x - \alpha \sqrt{\frac{q}{p}} + b^* \right) 
$$
intersects some lattice point, and the segment of this line such that $0
\le x \le \alpha \sqrt{q/p} - b^*$ is the hypotenuse of the triangle $T^*$.
Therefore, the lattice points enclosed by the triangle $T^*$ are exactly those
enclosed by the triangle with vertices $(0,0)$, $(\alpha \sqrt{q/p},0)$, and
$(0,\alpha \sqrt{p/q})$. It follows that the convex hull $C$ is contained in the
triangle $T^*$. We assert
that 
$$
T^* \setminus C \subset \left[0, q \right] \times \left[\alpha
\sqrt{\frac{p}{q}} - p - \frac{p}{q}b^*, \alpha \sqrt{\frac{p}{q}} - \frac{p}{q}
b^* \right] \cup \left[\alpha \sqrt{\frac{q}{p}} - q - b^*, \alpha
\sqrt{\frac{q}{p}} -b^*\right] \times \left[0, p \right].
$$
Indeed, the hypotenuse of $T^*$ intersects lattice points periodically because
it has a rational slope $-p/q$. Therefore, it must intersect exactly one lattice
point with $x$-coordinate $0 \le x < q$ and exactly one lattice point with
$y$-coordinate $0 \le y < p$ (since $p$ and $q$ are coprime). The line segment
between these two points must be contained in the convex hull $C$. We conclude
that the hypotenuse of $T^*$ and $C$ only possibly differ in the above union of
rectangles. Therefore, the area of $T^*$ and $C$ differ by at most
$\mathcal{O}_{p,q}(1)$, and moreover, the boundary of $T^*$ and the boundary of
$C$ contain the same number of lattice points up to error on the order of
$\mathcal{O}_{p,q}(1)$. Next, we use Pick's Theorem to estimate the number of
lattice points in the convex hull $C$. Let $A$ denote the area of $C$, $B$
denote the number of lattice points on the boundary of $C$, and $I$ denote the
number of lattice points in the interior of $C$. By Pick's Theorem
\begin{eqnarray*}
I + B &=& A + \frac{B}{2} + 1  \\
&=& \frac{\left( \alpha \sqrt{\frac{p}{q}} - \frac{p}{q} b^* \right)\left(\alpha
\sqrt{\frac{q}{p}} -
b^* \right)}{2} + 
\frac{ \alpha \sqrt{\frac{p}{q}} + \alpha \sqrt{\frac{q}{p}} +
\alpha\sqrt{\frac{1}{pq}}}{2} +
\mathcal{O}_{p,q}(1).
\end{eqnarray*}
Subtracting the lattice points on the $x$ and $y$ axes yields 
\begin{eqnarray*}
N_{p/q}(\alpha) &=& I + B - \alpha \sqrt{p/q} - \alpha \sqrt{q/p} +
\mathcal{O}_{p,q}(1) \\
&=& \frac{\alpha^2}{2} - b^*\alpha \sqrt{\frac{p}{q}} + \alpha
\frac{-\sqrt{\frac{p}{q}}
-\sqrt{\frac{q}{p}} + \sqrt{\frac{1}{pq}}}{2} + \mathcal{O}_{p,q}(1) \\
&=& \frac{\alpha^2}{2} + \alpha\frac{-\sqrt{\frac{p}{q}} - \sqrt{\frac{p}{q}} +
 \sqrt{\frac{1}{pq}} \left(1-2\{\alpha \sqrt{p q}\} \right)
 }{2} + \mathcal{O}_{p,q}(1).
\end{eqnarray*} 
The final step results from substituting $b^* = \{ \alpha \sqrt{pq} \} /p$, where
$\{\alpha \sqrt{pq}\} = \alpha \sqrt{pq} - \lfloor \alpha \sqrt{pq} \rfloor$
denotes the fractional part of $\alpha \sqrt{pq}$. This completes the proof.  
\end{proof}

\subsection{Irrational slopes: Lemma \ref{lem:irrational}}
To prove Lemma \ref{lem:irrational} we combine the results of Lemmas
\ref{lem:irrationalhull} and \ref{lem:irrationalarea}.
We note that those results are formulated
for triangles in a different configuration for simplicity of exposition; a
reflection and translation makes these results
applicable to triangles discussed in this proof.

\begin{proof}[Proof of Lemma \ref{lem:irrational}] 
Suppose a triangle with vertices $(0,0)$, $(\alpha/\sqrt{\beta},0)$, and
$(0,\alpha \sqrt{\beta})$ is given. Let $C$ denote the convex hull of the
nonnegative lattice points enclosed by this triangle. The number of lattice
points contained in this triangle is equal to the number $I+B$ of
lattice points contained in the convex hull $C$, which by Pick's
Theorem equals
$$
 I + B = A + \frac{B}{2} + 1,
$$
where $I$ is the number of lattice points in the interior of $C$, $B$ is the
number of lattice points on the boundary of $C$, and $A$ is the area of $C$.
Let 
$$
\gamma = \beta \left( \frac{\alpha}{\sqrt{\beta}} - \left\lfloor
\frac{\alpha}{\sqrt{\beta}} \right\rfloor \right)
\quad \text{and} \quad N = \left\lfloor \frac{\alpha}{\sqrt{\beta}} \right\rfloor.
$$
If the convex hull $C$ is reflected about the line $x = N$ and translated $N$
units to the left, then it will be in the configuration of Lemmas
\ref{lem:irrationalhull} and \ref{lem:irrationalarea} with $\gamma$ and $N$ as
specified above. Applying the results of these Lemmas, we conclude that
$$ 
B = (1 + \beta) N + o_\beta(N) = \left( \sqrt{\beta} + \sqrt{\frac{1}{\beta}}
\right) \alpha + o_\beta(\alpha),
$$
and
$$
A = \frac{1}{2} \beta N^2 + \gamma N + o_\beta(N) = \frac{\alpha^2}{2} +
o_\beta(N).
$$
Hence, the number $\# ( C \cap \mathbb{Z}^2_{>0})$ of positive lattice points in the
convex hull $C$ is equal to the number $I+B$ of nonnegative lattice points in
$C$ minus the number $(\sqrt{\beta}+\sqrt{1/\beta})\alpha+\mathcal{O}(1)$ of
lattice points in $C$ that are either on the $x$-axis or $y$-axis 
\begin{eqnarray*}
\# \left( C \cap \mathbb{Z}_{>0}^2 \right) &=& I + B - \left( \sqrt{\beta} +
\sqrt{\frac{1}{\beta}} \right)\alpha + \mathcal{O}(1) \\ 
&=& \alpha^2/2 - \frac{  \sqrt{\beta} + \sqrt{1/\beta}}{2}\alpha +
o_{\beta}(\alpha),
\end{eqnarray*}
which is the desired statement.
\end{proof}

\subsection{Minkowski dimension.}
The discussion following Lemma \ref{lem:irrational} implies that $\Lambda$ only
contains rational numbers. We can use the asymptotic formula in  Lemma
\ref{lem:rational} to show that the rational numbers $p/q$ in $\Lambda$ have to
be increasingly close to 1 as the denominator increases. This allows us to prove
the following Lemma.
\begin{corollary}
If $p/q \in \Lambda$, then $|q-p| \leq 2 \sqrt{q} + 1$. Moreover, $\dim \Lambda
\le 3/4$.
\end{corollary}

\begin{proof}
From Lemma \ref{lem:rational} (rational slopes), we see that $p/q \in \Lambda$ implies
$$ - \sqrt{\frac{p}{q}} - \sqrt{\frac{q}{p}} + \sqrt{\frac{1}{pq}} \geq -2.$$
Multiplying with $\sqrt{pq}$ yields
$$ -p - q+ 1 \geq -2 \sqrt{pq}.$$
This yields a quadratic inequality with equality exactly for $p=q \pm 2\sqrt{q} + 1$.
It now suffices to compute the Minkowski dimension
$$\dim_{} \left\{ \frac{p}{q} \in \mathbb{Q}: q \geq 1 \wedge |p - q| \leq 2\sqrt{q} + 1\right\}.$$
These rational numbers accumulate around 1. It is easy to see that, for every $\varepsilon > 0$,
$$  \left\{ \frac{p}{q} \in \mathbb{Q}_{>0}:  |p - q| \leq 2\sqrt{q} +
1\right\} \subseteq \left[ 1 - \varepsilon^{1/4}, 1 + \varepsilon^{1/4} \right]
\cup  \left\{ \frac{p}{q} \in \mathbb{Q}_{>0}: q \leq
\frac{9}{\sqrt{\varepsilon}} \wedge |p - q| \leq 2\sqrt{q} + 1\right\}.$$
Covering the first set with $\varepsilon$-boxes requires $\sim \varepsilon^{-3/4}$ boxes. As for the second set, we simply put a box around every
element which puts an upper bound on the number of boxes required at
$$ \# \left\{ \frac{p}{q} \in \mathbb{Q}: q \leq \frac{9}{\sqrt{\varepsilon}}
\wedge |p - q| \leq 2\sqrt{q} + 1\right\} \lesssim \sum_{k=1}^{9 \varepsilon^{-1/2}}{ \sqrt{k}} \sim  \left(\varepsilon^{-1/2}\right)^{3/2} \sim \varepsilon^{-3/4}.$$
\end{proof}

\subsection{Proof of Theorem \ref{thm1}}
\begin{proof} We quickly re-iterate why no irrational slope can be optimal for a sequence
of areas tending to infinity. Lemma \ref{lem:irrational} implies that
$$
N_{\beta}(\alpha) = \frac{\alpha^2}{2} - \alpha\frac{\sqrt{\beta} +
\sqrt{\frac{1}{\beta}}}{2} + o_\beta(\alpha)
$$
and, since $1 \in \mathbb{Q}$, we have 
$$ \frac{\sqrt{\beta} +\sqrt{\frac{1}{\beta}}  }{2}  > 1,$$
which means that the asymptotic number of lattice points is eventually dominated by
the rational slope $(n+1)/n$ for $n$ sufficiently large (see the discussion
following Lemma \ref{lem:irrational}). The very same reason,
combined with Lemma \ref{lem:rational}, shows that
the limit set $\Lambda$ can only contain rational slopes with 
$$ -\sqrt{\frac{p}{q}} - \sqrt{\frac{q}{p}} + \sqrt{\frac{1}{pq}} \geq -2.$$
This can be used to show that $\Lambda \subset [1/3,3]$: if $p/q \geq 3$ and
$p/q \in \Lambda$, then
$$ -\sqrt{\frac{p}{q}} - \sqrt{\frac{q}{p}} \leq -2.3 \quad \mbox{and thus} \quad  \sqrt{\frac{1}{pq}} \geq 0.3.$$
This last inequality is only true for finitely many rational numbers that can be explicitly checked by hand. The
case $p/q \leq 1/3$ follows from symmetry considerations.
A similar argument establishes the nontrivial dynamics: suppose it were indeed the case that the finite set of
slopes
$$\Gamma = \left\{ \frac{p_1}{q_1}, \frac{p_2}{q_2}, \dots, \frac{p_n}{q_n} \right\} \subset \Lambda$$
captures more lattice points for all sufficiently large areas then any other triangle whose slope is not in the set.
We know that the counting function $N_{p/q}$ is oscillating periodically around a limit
value and has (relatively) small values in a periodically occuring manner. More precisely, for every $p/q \in \mathbb{Q}$
there exists $\varepsilon_{p,q} > 0$ and $\delta_{p,q} > 0$ such that for all $k \in \mathbb{N}$
$$ \forall \alpha \in \left(\frac{k}{\sqrt{p q}} - \varepsilon_{p,q}, \frac{k}{\sqrt{p q}} \right): \qquad N_{p/q}(\alpha) \leq \alpha^2/2 - (1+\delta_{p,q})\alpha.$$
If the limiting set is finite, then we can obtain a uniform $\delta > 0$ that is
valid for all elements of $\Gamma$ and use Lemma \ref{basicalign} to very nearly
align the location of the minima. Comparing with slope $(n+1)/n$ for $n$
sufficiently large (depending only on $\delta$) then yields a contradiction.
In fact, a stronger alignment result is proved in Lemma \ref{align}, from which
the conclusion that $\Lambda$ is infinite immediately follows -- this will be explained in greater detail and used at the end of the paper.
\end{proof}

\subsection{Proof of the Proposition \ref{prop1} }
\begin{proof}
We want to show $N_1(n) > N_\beta(n)$ for all $\beta \not = 1$ and all $n \in
\mathbb{N}$. Consider the triangle with vertices $(0,0)$, $(x,0)$, and $(0,y)$
satisfying $x y = n^2$ and without loss of generality $y \ge x$ (and thus $y\geq n$). First, observe
that
$$
N_1(n) = \frac{n(n-1)}{2} = \frac{n^2}{2} - \frac{n}{2}.
$$
Consider the convex hull $C$ of the nonnegative lattice points enclosed by the
triangle with vertices $(0,0)$, $(x,0)$, and $(0,y)$. Let $I$ and $B$ denote the
number of lattice points in the interior and on the boundary of $C$,
respectively, and let $A$ denote the area of $C$. By Pick's Theorem
$$
I + B = A + \frac{B}{2} + 1.
$$
The number of points $B$ on the boundary of $C$ can be written
$$
B = \lfloor x \rfloor  + \lfloor y \rfloor + D + 1,
$$
where $D$ denotes the number of (strictly) positive lattice points on the
boundary of $C$. Then
$$
N_{y/x}(n) = I + D = A +  \frac{-\lfloor x \rfloor - \lfloor y \rfloor + D +
1}{2} .
$$
The area $A \le x y /2 = n^2/2$, and $D \le \lfloor x \rfloor$.
Therefore 
$$
N_{y/x}(n) \le \frac{n^2}{2} - \frac{ \lfloor y \rfloor + 1}{2} .
$$
If $\lfloor y \rfloor \ge n + 1$, then the result immediately follows from this
inequality.  Otherwise, if $n < y < n +1$, then the triangle with vertices
$(0,0), (x,0)$ and $(0,y)$ does not intersect the line of slope $-1$, which
intersect the $y$-axis at $n+1$. The number of positive lattice points under
this line is exactly $n^2/2 - n/2$ so we conclude
$$
N_{y/x}(n) \le \frac{n^2}{2} - \frac{n}{2},
$$
which completes the proof.
\end{proof}

\section{Lemmas for the Proof of Theorems \ref{limitset} and \ref{badarea}}
\label{sec:lemmas}

We begin by proving three lemmas that further characterize good slopes, and one
lemma that characterizes the dynamics of integral multiples of radicals. These
lemmas strengthen previous lines of reasoning, and together lead to proofs of Theorem \ref{limitset} and \ref{badarea}. Throughout
this section we use the
notation 
$$
f(x) \lesssim_h g(x) \iff f(x) \le C_h g(x),
$$
for all $x$ where $C_h$ is a constant only depending on $h$.

\subsection{Good slopes have many positive lattice points on their convex hull's
boundary} First, we quantify the notion of a good slope by asking that the
number of lattice points captured by such slopes exceeds $\alpha^2/2 - \alpha$
by a term that is linear in $\alpha$.  Specifically, for any $\gamma > 0$ we say
that $\beta > 0$ is a $\gamma$-good slope at area $\alpha > 0$ provided 
$$ 
N_\beta(\alpha) > \frac{\alpha^2}{2} - \alpha + \frac{\gamma}{2}\alpha.
$$
 Let $C$ denote the convex hull of the nonnegative lattice points
enclosed by the triangle with vertices $(0,0)$, $(\alpha/\sqrt{\beta})$, and
$(0,\alpha \sqrt{\beta})$. To be clear, by points enclosed by the triangle we
mean the set of points in the interior or on the boundary of the triangle. Let 
$$
d_\beta(\alpha) = \frac{\# \{ n \in \partial C \cap \mathbb{Z}_{>0}^2
\}}{\alpha}
$$
denote the number of positive lattice points on the boundary of the convex hull
$C$ divided by $\alpha$.  We will show that if $\beta$ is a $\gamma$-good slope
for an area $\alpha$, then, if $\alpha$ is sufficiently large, $d_\beta(\alpha)
\gtrsim \gamma$ .

\begin{lemma} \label{curve} 
For all $\gamma > \eta > 0$ and all sufficiently large areas (where sufficiently
large depends only on $\gamma, \eta$) we have $$ \beta \text{ is }
\gamma\text{-good at } \alpha \implies
 d_\beta(\alpha) > \eta \quad \text{and} \quad
1/4 \le \beta \le 4.
$$
\end{lemma}

\begin{proof}
Suppose $\beta > 0$ is given. Let $C$ denote the convex hull of the nonnegative
lattice points enclosed by the triangle with vertices $(0,0),$
$(\alpha/\sqrt{\beta},0)$, and $(0,\alpha \sqrt{\beta})$. We have
$$
N_\beta(\alpha) = I + D,
$$
where $I$ is the number of lattice points in the interior of $C$, and $D$ is the
number of (strictly) positive lattice points on the boundary of $C$. If $X$ and
$Y$ denote the number of lattice points on the boundary of $C$ and on the
$x$-axis and $y$-axis, respectively, then Pick's Theorem yields the alternative
representation 
$$
N_\beta(\alpha) = A + \frac{-X - Y + D}{2} + \mathcal{O}(1),
$$
where $A$ is the area of $C$. Bounding $A \le \alpha^2/2$ and
substituting explicit expressions for $X$, $Y$, and $D$ into the above equation
gives
$$ 
N_{\beta}(\alpha) \leq \frac{\alpha^2}{2} + \alpha \frac{-\sqrt{\beta} -
\sqrt{\frac{1}{\beta}} + d_\beta(\alpha)}{2} + \mathcal{O}(1).
$$
Choose $M_{\gamma,\eta} > 0$ such that $(\gamma-\eta) M_{\gamma,\eta}/ 2$ is
greater than the implicit constant in the above expression. Then for all $\alpha
> M_{\gamma,\eta}$
$$
N_{\beta}(\alpha) > \frac{\alpha^2}{2} - \alpha  + \frac{\gamma}{2} \alpha
\implies \frac{-\sqrt{\beta} - \sqrt{\frac{1}{\beta}}}{2} +
\frac{d_\beta(\alpha)}{2} > -1 + \frac{\eta}{2}.
$$
Since for all $\beta > 0$, $\sqrt{\beta} + \sqrt{1/\beta} \ge 2$
we conclude
$$
d_\beta(\alpha) > \eta.
$$
The statement $1/4 \leq \beta \leq 4$ follows as above: if $\beta > 4$, then already the linear term shows that the result cannot hold. (This result
could be improved to $1/3 \leq \beta \leq 3$ as indicated above but this is not necessary here, any explicit bound suffices for our purposes).
\end{proof}

\subsection{Good slopes are close to rational slopes} We now establish that
any $\gamma$-good slope $\beta$ must be close to a rational number $p/q$ with
denominator $q \lesssim 1/\gamma$.

\begin{lemma} \label{closetorational}
For all $\gamma > \eta > 0$ and all sufficiently large areas (depending on
$\gamma, \eta$) we have that for all $\beta \ge 1$
$$ 
\text{$\beta$ is $\gamma$-good} \quad \implies \quad 
\exists \frac{p}{q}\in \mathbb{Q}
\cap \left[1, 4\right] \quad \text{such that} \quad \left|\beta -
\frac{p}{q} \right| \lesssim_{\gamma,\eta} \frac{1}{\alpha} 
\quad \text{and} \quad q < \frac{1}{\eta}.
$$
\end{lemma}

\begin{proof} We quickly summarize the idea behind the proof before giving technical
details: if the curved part of the boundary of the convex hull has many points, then many
of the slopes that arise have to be rational with a small denominator; a pigeonhole argument
shows that one of these has to occur for a long stretch: the true slope $\beta$
must closely match the rational number over that long stretch, otherwise the convex hull
would look differently.
Let $\gamma > \eta > 0$ be given, and suppose that $\beta > 0$ is $\gamma$-good
for area $\alpha >0$. Since we have assumed $\beta \ge 1$, applying Lemma
\ref{curve} for $\eta^\prime = \sqrt{\eta \gamma} > 0$ gives
$$
d_\beta(\alpha) > \eta^\prime \qquad \mbox{and} \qquad 1 \le \beta \le 4.
$$
Let $C$ denote the convex hull of the nonnegative lattice points enclosed by the
triangle with vertices $(0,0)$, $(\alpha/\sqrt{\beta},0)$, and $(0,\alpha
\sqrt{\beta})$. We call $\overline{\partial C \cap \mathbb{R}^2_{>0}}$ the
`curved' part of the boundary of the convex hull $C$. Let $S_i$ denote the line
segment from $(x_i,y_i)$ to $(x_i+q_i n_i,y_i-p_i n_i)$ where
$x_i,y_i,q_i,p_i,n_i \in \mathbb{Z}_{\ge 0}$ and $p_i$ and $q_i$ are coprime.
Formally,
$$
S_i = \left\{ (x,y) : y = -\frac{p_i}{q_i} (x - x_i) + y_i \quad \text{for}
\quad x_i \le x \le q_i n_i + x_i \right\}.
$$
The curved part of the convex hull can be expressed as a
union of such line segments
$$
\overline{\partial C \cap \mathbb{R}^2_{>0}}  = \bigcup_{i=1}^m \left\{ (x,y) :
y = -\frac{p_i}{q_i} (x - x_i) + y_i \quad \text{for} \quad x_i \le x \le q_i
n_i + x_i \right\},
$$
where the sequences $(x_i)_{i=1}^m$ and $(p_i/q_i)_{i=1}^m$ are strictly
increasing. The reason that we may assume that $p_i/q_i$ is strictly increasing
with $x_i$ is that a decrease in $p_i/q_i$ would violate the convexity of $C$,
and adjoining segments of equal slope can be grouped into a single segment. Note
that since we assumed $p_i$ and $q_i$ are coprime, each line segment can be
decomposed into $n_i$ smaller segments which intersect lattice points at their
endpoints, but not in their interiors. With this notation,
$$
\sum_{i=1}^m n_i= \alpha d_\beta(\alpha) + 1 \geq \eta^\prime \alpha.
$$
We have assumed $\beta \ge 1$ as a hypothesis to the Lemma (However, the Lemma
applies equally well to slopes less than $1$ by flipping the entire triangle
around the $y=x$ diagonal and considering slopes $q/p$ and $1/\beta$ instead).
The assumption $\beta \ge 1$ implies that the length of the side of the triangle
on the $x$-axis is at most $\alpha$ and 
$$
\alpha \ge \sum_{i=1}^n q_i n_i .
$$
Multiplying by $\eta^\prime$ and applying Markov's inequality for a parameter
$\lambda > 0$ gives
$$
\eta^\prime \alpha \ge \eta^\prime \sum_{i=1}^m  q_i n_i \ge \eta^\prime \lambda \sum_{q_i \ge \lambda} n_i.
$$
Multiplying by $-1/(\eta^\prime \lambda)$, adding $\sum_{i=1}^m
n_i$, and using $\sum_{i=1}^m n_i \ge \eta^\prime \alpha$ gives
$$
\sum_{q_i < \lambda} n_i \ge \sum_{i=1}^m n_i - \frac{ \alpha}{\lambda} \ge
\left(\eta^\prime -\frac{1}{\lambda} \right) \alpha.
$$
Setting $\lambda = \gamma/(\eta^\prime)^2$ yields
$$
\sum_{q_i < \gamma/(\eta^\prime)^2} n_i \ge \frac{\gamma - \eta^\prime}{\gamma}
\eta^\prime \alpha.  
$$
Substituting $\eta^\prime = \sqrt{\gamma \eta}$ and using the fact that $q_i \ge
1$ gives
$$
\sum_{q_i < 1/\eta} q_i n_i \ge \sum_{q_i < 1/\eta} n_i \ge c \alpha \quad
\text{where} \quad c = \quad \frac{\gamma - \sqrt{\gamma \eta}}{\gamma}
\sqrt{\gamma \eta}.
$$
We assert that
$$
\sum_{q_i < 1/\eta \, :\, 1 \le p_i/q_i \le 4} q_i n_i \ge \frac{c \alpha}{2}.
$$
Indeed, otherwise either
$$
\sum_{q_i < 1/\eta\, :\, p_i /q_i \le 1 - \eta} q_i n_i \ge \frac{c \alpha}{4}
\quad \text{or} \quad
\sum_{q_i < 1/\eta \, :\, p_i /q_i \ge 4 + \eta} q_i n_i \ge \frac{c \alpha}{4}.
$$
Since $1 \le \beta \le 4$, either case would imply that a part of the convex
hull of length greater than $c \alpha/4$ consists of slopes that are either all
less than $\beta$ by $\eta$ or all greater than $\beta$ by $\eta$. In either
case, when $\alpha$ is sufficiently large (depending on $\gamma, \eta$), this
leads to a contradiction because these parts of the convex hull would deviate
from the line of slope $\beta$ by more than $1$.  Thus, informally speaking, we
have that at least a constant proportion (determined by $\gamma, \eta$) of the
curved part of the boundary of the convex hull $C$ consists of segments with
rational slopes contained in $[1,4]$ with denominators less than $1/\eta$. The
number of such slopes is
$$
\# \left\{ \frac{p_i}{q_i} \in \mathbb{Q} \cap \left[1,4 \right] : q <
1/\eta \right\} \le 4/\eta^2.
$$
Therefore, by the pigeonhole principle there exists a slope $p_i/q_i$ such that
$$
\frac{p_i}{q_i} \in \mathbb{Q} \cap \left[ 1,4 \right] \quad
\text{such that} \quad q_i < \frac{1}{\eta} \quad \text{and} \quad q_i n_i
\ge 
\frac{c \eta^2}{8} \alpha.  
$$
We emphasize that since slopes of line segments on the boundary of the convex
hull are monotone, these $n_i$ segments of length $q_i$ are next to each other,
which means that the convex hull has a very long line segment of a fixed
rational slope in $[1,4]$ whose denominator is less than $q_i$. The length
of the projection of this line segment on the $x$-axis is greater than $(c
\eta^2/8) \alpha$. The difference in the height change of
the line of slope $-\beta$ and this line segment of slope $-p_i/q_i$ must be
less than $1$
$$
\left| (\beta  - p_i/q_i) \frac{c \eta^2}{8} \alpha \right| \le 1.
$$
Moving the term $(c \eta^2/8) \alpha$ to the right hand side yields the result,
as recall that the constant $c > 0$ depends only on $\eta$ and $\gamma$.
\end{proof}

\subsection{Slopes near poorly performing rational slopes cannot perform well}

In the following we show that if a rational slope performs poorly at a specific
area, then any nearby slope cannot perform particuarlly well at the same area.
that any nearby slope associated to the same area cannot preform particularly
well. This result may be regarded as a type of stability statement. Recall that
by Lemma \ref{lem:rational}
$$
N_{p/q}(\alpha) = \frac{\alpha^2}{2} + 
\frac{-\sqrt{\frac{p}{q}}-\sqrt{\frac{q}{p}}+\sqrt{\frac{1}{p q}}(1-2
\{\alpha \sqrt{p q} \})}{2} \alpha + \mathcal{O}_{p,q}(1).
$$
For large $\alpha$ (depending on $p$ and $q$) the performance of a rational
slope is worst when 
$$
\{\alpha \sqrt{pq} \} \in (1-\varepsilon,1),
$$
for some small $\varepsilon > 0$. We show that at such $\alpha$ any slope
$\beta$ close to $p/q$ cannot perform very well.

\begin{lemma} \label{mimic}
Suppose $p$ and $q$ are coprime positive integers, and $c > 0$ is fixed. Then,
for all $\varepsilon > 0$,
$$
\{\alpha \sqrt{pq} \} \in (1-\varepsilon,1) \quad \text{and} \quad \left|\beta -
\frac{p}{q} \right| < \frac{c}{\alpha} \implies N_{\beta}(\alpha)  \le
\frac{\alpha^2}{2} - \alpha + \mathcal{O}_{p,q}( 1 + \varepsilon \alpha),
$$
and furthermore,
$$
\{\alpha \sqrt{pq} \}  = 0 \quad \text{and} \quad 0 < \left|\beta - \frac{p}{q}
\right| < \frac{c}{\alpha} \implies N_{\beta}(\alpha)  \le \frac{\alpha^2}{2} -
\alpha + \mathcal{O}_{p,q}( 1 ).
$$

\end{lemma}
\begin{proof}

Suppose $\{ \alpha \sqrt{p q} \} \in (1-\varepsilon,1)$ where $\varepsilon > 0$.
By Lemma \ref{lem:rational}, the function $N_{p/q}(\alpha)$ satisfies
$$
 N_{p/q}(\alpha) = \frac{\alpha^2}{2} + \frac{ -\sqrt{\frac{p}{q}} -
\sqrt{\frac{q}{p}} -  \sqrt{\frac{1}{pq}}}{2} \alpha +
\mathcal{O}_{p,q}(1 + \varepsilon \alpha). 
$$
We are interested in understanding $N_\beta(\alpha)$ for slopes $\beta$ close
to $p/q$.  The situation is clarified by considering the set of all lines of
slope $-p/q$ which intersect a lattice point. These lines intersect the $x$-axis
and $y$-axis periodically with period $1/p$ and $1/q$, respectively.  Moreover,
each line intersects one lattice point for every $q$ units traveled in the
$x$-direction. There is a rather clean and unifying picture (see Figure
\ref{fig:manylines}) with three main components. First, we draw a dashed line
representing the hypotenuse of the triangle associated with slope $-p/q$ and
area $\alpha^2/2$.  Second, we draw all lines of slope $-p/q$ that intersect a
lattice point, and label these lines by $\ell_j$ for $j \in \mathbb{Z}$.  Third,
we draw a bold line of slope $-\beta$ associated with a triangle of area
$\alpha^2/2$; this line represents the result of tilting the
dashed line a bit.
\begin{figure}[h!]
\begin{tikzpicture}[scale=.5]
\draw[->] (0,0) -- (0,13);
\draw[->] (0,0) -- (13,0);
\draw [thick]  (12,0) -- (0,12);
\draw [thick]  (11,0) -- (0,11);
\draw [thick]  (10,0) -- (0,10);
\draw [thick]  (0,0) -- (0,10);
\draw [thick]  (9,0) -- (0,9);
\draw [dashed]  (8.8,0) -- (0,8.8);
\draw [thick]  (8,0) -- (0,8);
\draw [thick]  (7,0) -- (0,7);
\draw [thick]  (6,0) -- (0,6);
\draw [ultra thick]  (7,0) -- (0,11);
\node at (-0.6, 6) {$\ell_{-3}$};
\node at (-0.6, 7) {$\ell_{-2}$};
\node at (-0.6, 8) {$\ell_{-1}$};
\node at (-0.4, 9) {$\ell_0$};
\node at (-0.4, 10) {$\ell_1$};
\node at (-0.4, 11) {$\ell_2$};
\node at (-0.4, 12) {$\ell_3$};
\end{tikzpicture}
\caption{The hypotenuse of the rational triangle (dashed), all parallel lines
of the same slope that intersect a lattice point (solid), and the hypotenuse
of a triangle with a nearby slope (bold).} \label{fig:manylines} 
\end{figure}
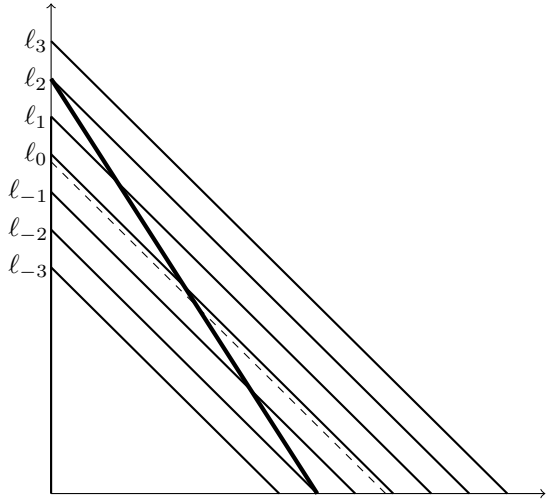

We make two remarks about Figure \ref{fig:manylines}. First, observe that the
dashed line occurs directly below a line of slope $-p/q$ which intersects a
lattice point; this corresponds to the fact that $\{\alpha \sqrt{pq} \} \in
(1-\varepsilon,1)$ since the hypotenuse will intersect a lattice point when
$\alpha \sqrt{p q}$ is an integer. Second, note that each line of slope $-p/q$
which intersects a lattice point, intersects a lattice point periodically
(intersecting one lattice point for every $q$ units on the $x$-axis); thus, up
to errors of order $\mathcal{O}(1)$, the number of lattice points captured from
such a line is proportional to the length of the captured line segment. In the
context of Figure \ref{fig:manylines} let us consider the result of tilting the
dashed line of slope $-p/q$ in terms of the number of captured lattice points.
Relatively quickly (and depending on $\varepsilon$), we capture half the
lattice points on $\ell_0$ without losing any lattice points. However, this is
basically the best that we can do: afterwards, when we start gaining lattice
points on $\ell_1$, we simultaneously lose about the same number of lattice
points on $\ell_{-1}$ and, more generally, lose about the same number on
$\ell_{-i}$ as we gain on $\ell_{i}$.  It remains to make this notion precise.
We first consider the lattice points gained from the line $\ell_0$. The
total number of lattice points on the line $\ell_0$ is $\alpha /\sqrt{pq}$ and
we add at most half of them. Observe that 
$$
N_{p/q}(\alpha) + \frac{1}{2} \alpha \sqrt{\frac{1}{p q}}  \le
\frac{\alpha^2}{2} - \alpha \frac{ \sqrt{\frac{p}{q}} + \sqrt{\frac{q}{p}}}{2} +
 \mathcal{O}_{p,q}(1 + \varepsilon \alpha) \le \frac{\alpha^2}{2} - \alpha +
\mathcal{O}_{p,q}(1 + \varepsilon \alpha),
$$
and therefore adding half the points on $\ell_0$ does not violate the bound of
the Lemma. Second, we note that since $|\beta - p/q| < c/\alpha$ where $c > 0$
is a fixed constant, it follows that the line of slope $\beta$ can only
intersect $O_{p,q}(1)$ lines $\ell_j$. Thus, it suffices to show that the net
change in the number of lattice points resulting from the intersection of our
tilted line with the lines  $\ell_j$ and $\ell_{-j}$ is $\mathcal{O}_{p,q}(1+
\varepsilon \alpha)$. The equation of the line of slope $-\beta$, which
intersects the $y$-axis
at $\alpha \sqrt{\beta}$ is
$$
y = -\beta x + \alpha \sqrt{\beta}.
$$
The family of lines of slope $-p/q$ that intersect a lattice point are given by
the equation
$$
y = -\frac{p}{q} x + (\alpha + \zeta) \sqrt{\frac{p}{q}} + \frac{j}{q},
$$
where $j \in \mathbb{Z}$ and $0 < \zeta < \varepsilon/\sqrt{pq}$ since $\{
\alpha \sqrt{pq} \} \in (1-\varepsilon,1)$.  We will
refer to the line of parameter $j$ as $\ell_j$.  The $x$-coordinate $x_j$ of the
intersection of $\ell_j$ with the line of slope $-\beta$ is
$$
x_j = \frac{\alpha \sqrt{\beta} - (\alpha + \zeta) \sqrt{\frac{p}{q}} - \frac{j}{q}}{\beta - \frac{p}{q}}.
$$
We now add the number of lattice points gained from intersecting
the line $\ell_j$ and subtract those lost from intersecting $\ell_{-j}$ (see Figure
\ref{fig:manylines}).
The net change in lattice points on the lines $\ell_j$ and $\ell_{-j}$ is equal
to $1/q$ times 
$$
 x_j - \left( \sqrt{\frac{q}{p}} \alpha - \frac{j}{p} - x_{-j} \right)   +
\mathcal{O}_{p,q}(1) =  2 \alpha \frac{\sqrt{\beta} - \sqrt{\frac{p}{q}}}{\beta
- \frac{p}{q}}- \sqrt{\frac{q}{p}} \alpha  - 2 \frac{ \zeta
\sqrt{\frac{p}{q}}}{\beta - \frac{p}{q}}  + \frac{j}{p} + \mathcal{O}_{p,q}(1).
$$
Now we estimate the first and second term on the right hand side of the above
equation 
$$ 
2 \frac{\sqrt{\beta} - \sqrt{\frac{p}{q}}}{{\beta - \frac{p}{q}}} \alpha - \sqrt{\frac{q}{p}}
\alpha = \frac{2 \alpha}{\sqrt{\beta} + \sqrt{\frac{p}{q}}} - \sqrt{\frac{q}{p}}
\alpha \le \frac{2 \alpha}{2 \sqrt{\frac{p}{q}}} - \sqrt{\frac{q}{p}} \alpha +
\mathcal{O}_{p,q}(1) = \mathcal{O}_{p,q}(1).
$$
Since the line of slope $\beta$ intersects $\mathcal{O}_{p,q}(1)$ lines
$\ell_j$, the term $j/p$
is $\mathcal{O}_{p,q}(1)$. It remains to estimate the term $2 \zeta
\sqrt{p/q}/(\beta - p/q)$. The key observation is that if the line of slope
$\beta$ intersects the line $\ell_j$ (it suffices consider the line $\ell_1$)
then it must be tilted enough such that $|\beta-p/q|$ is not that small, in
particular, 
$$
\frac{1}{|\beta - p/q|} = \mathcal{O}_{p,q}(\alpha).
$$
Since $0 < \zeta < \varepsilon/\sqrt{pq}$ we conclude the term $2 \zeta
\sqrt{p/q}/(\beta - p/q)$ is $\mathcal{O}_{p,q}(\varepsilon \alpha)$. Thus, in
combination, we have 
$$
 x_j - \left( \sqrt{\frac{q}{p}} \alpha - \frac{j}{p} - x_{-j} \right)   +
\mathcal{O}_{p,q}(1) = \mathcal{O}_{p,q}(1+\varepsilon \alpha ).
$$
This establishes the first statement of the Lemma.  If $\{ \alpha \sqrt{p q} \}
= 0$ and  $0 < |\beta - p/q| < c/\alpha$, then we still capture half of the
points on the line $\ell_0$ so the first part of the proof is unchanged.
Furthermore, the analysis of the net change from the intersections with the
lines $\ell_j$ and $\ell_{-j}$ is stable as $\varepsilon \rightarrow 0$, so the
second statement of the Lemma follows from an identical argument to the first.
\end{proof}

\subsection{Aligning Radicals}
In this section we establish an alignment result for the fractional part of
integer multiples of radicals of square-free integers.  We say $n \in
\mathbb{N}$ is square-free provided $n$ can be expressed as the product of
distinct prime numbers. Furthermore, we say irrational numbers
$v_1,v_2,\ldots,v_m$ are linearly independent over the rationals if
$$
(v_1,v_2,\ldots,v_m) \cdot n \not = 0, \quad \forall n \in \mathbb{Z}^m
\setminus \{0\}.
$$
The key idea used to establish the alignment result in this section is the
following result of Besicovitch \cite{besicovitch} (several different proofs are
given by Boreico \cite{Boreico2008}).

\begin{theorem}[Besicovitch \cite{besicovitch}] \label{radical}
Suppose $n_1,n_2,\ldots,n_m$ are distinct square-free integers. Then
$$
\sqrt{n_1},\sqrt{n_2},\ldots,\sqrt{n_m},
$$
are linearly independent over the rationals.
\end{theorem}

As usual, given a vector $v = (v_1,v_2,\ldots,v_m) \in \mathbb{R}^m$, define the vector of
fractional parts $\{v\} \in \mathbb{R}^m$ by
$$
\{ v \} = ( v_1 - \lfloor  v_1 \rfloor,  v_2 - \lfloor
 v_2 \rfloor, \ldots,  v_m - \lfloor  v_m \rfloor ).
$$
This Theorem can be combined with Kronecker's
Theorem \cite{Kronecker1884}: it says that if $v \in \mathbb{R}^m$ is a vector whose entries
combined with 1 are linearly independent over $\mathbb{Q}$, then
$$  \left\{k v \right\}_{k \in \mathbb{N}} \quad \mbox{is uniformly distributed in}~[0,1]^m.$$
We will not require uniform distribution, we shall only use that uniformly distributed sequences are dense.
Our main ingredient is the following.

\begin{lemma} \label{align}
Suppose $n_1,n_2,\ldots,n_{m-1} \in \mathbb{N}$ are distinct square-free numbers bigger than 1 and
assume that $n_m$ is a prime number that does not divide any of the $n_1, \dots, n_{m-1}$. 
Then there exists infinitely many numbers of the form $\alpha = k \sqrt{n_m}$ for $k \in \mathbb{N}$ such that $$\{
\alpha (\sqrt{n_1}, \dots, \sqrt{n_{m-1}}) \} \in (1-\varepsilon,1)^{m-1}.$$ 
\end{lemma}

\begin{proof} Since $n_m$ is a prime that does not divide any of the other numbers, the list
$$ 1, ~n_m, ~n_1 n_m,~ n_2 n_m, ~\dots, ~n_{m-1} n_m$$
is a list of distinct square-free numbers. By the Theorem of Besicovitch
\cite{besicovitch}, their roots are linearly independent and thus, by
Kronecker's theorem, the sequence $$ k \left(\sqrt{n_m}, \sqrt{n_m n_1},
\sqrt{n_m n_2}, \dots, \sqrt{n_m n_{m-1}} \right) \qquad \mbox{is uniformly
distributed.}$$
As a byproduct, the sequence is dense and there exists a subsequence that is contained in $[0,1] \times (1-\varepsilon, 1)^{m-1}$
from which the result then follows since
$$ k \left(\sqrt{n_m}, \sqrt{n_m n_1}, \sqrt{n_m n_2}, \dots, \sqrt{n_m n_{m-1}} \right)  = k \sqrt{n_m} \left(1, \sqrt{n_1}, \sqrt{n_2}, \dots, \sqrt{n_{m-1}} \right).$$
\end{proof}

\section{Proof of Theorem \ref{limitset}: there is an infinite subset of $\Lambda$ in $S$}
\label{sec:proofthm23}

This section is devoted to establishing Theorem \ref{limitset}: the limit set
$$ 
S =  \bigcap_{r >0 } \bigcup_{\alpha > r}{ \argmax_{\beta >0} N_\beta(\alpha)}
$$ 
contains infinitely many elements from $\Lambda$. 

\subsection{Warming up.}
The proof follows essentially
from using all the Lemmas in the right order, however, since this is rather lengthy
we start by giving a much simpler statement first. It has the advantage of being
quite transparent and demonstrating the outline of the argument.

\begin{proposition} We have
$3/2 \in \Lambda \cap S$.
\end{proposition}

\begin{proof} Recall that, for any $\gamma > 0$, we say
that a slope $\beta > 0$ is a $\gamma-$good at area $\alpha > 0$ if 
$$ 
N_\beta(\alpha) > \frac{\alpha^2}{2} - \alpha + \frac{\gamma}{2} \alpha.  
$$

We start by remarking that $3/2$ is $\gamma-$good at areas $\alpha = k \sqrt{6}$
for $k \in \mathbb{N}$ for
$$
\gamma = 0.36 < 2 - \sqrt{\frac{3}{2}} - \sqrt{\frac{2}{3}} +
\sqrt{\frac{1}{6}} .
$$
Without loss of generality (by symmetry) we may restrict our consideration to
slopes $\beta \ge 1$.  We fix this value of $\gamma$ and use Lemma
\ref{closetorational} with $\eta = 1/3$ to conclude
$$ 
\beta~\mbox{is}~\gamma-\mbox{good} \implies \exists \frac{p}{q}\in \mathbb{Q}
\cap \left[1, 5\right] \quad \text{such that} \quad \left|\beta -
\frac{p}{q} \right| \lesssim  \frac{1}{\alpha} 
\quad \text{and} \quad q < 3.
$$
 Then there is a finite list of slopes whose denominator is
less than 3 and this list is given by
$$
G = \left\{ 1,2,3,4,5,\frac{3}{2},\frac{5}{2},\frac{7}{2},\frac{9}{2}\right\}.
$$
Put differently, for area $\alpha = k \sqrt{6}$ with $k \in \mathbb{N}$, the
only slopes that could potentially be better than $3/2$ are close to an element
in $G$:
$$
\beta~\mbox{is}~\gamma-\mbox{good} \implies \exists \frac{p}{q}\in G \quad \text{such that} \quad \left|\beta -
\frac{p}{q} \right| \lesssim  \frac{1}{\alpha}.
$$
We will now construct areas of the form $\alpha = k\sqrt{6}$ such that most slopes in $G$ perform
pretty badly and will then use the stability statement of Lemma \ref{mimic} to conclude that nearby slopes
are not performing particularly well either.
By Lemma \ref{lem:rational} the behavior of a rational slope $p/q$ is determined
by $\{ \alpha \sqrt{p q} \}$ for large areas $\alpha$. Therefore, we consider the
set
$$
A = \{ \sqrt{p  q} : \frac{p}{q} \in G \wedge \text{$p,q$ coprime} \}.
$$
Explicitly,
$$
A = \left\{ 1,\sqrt{2},\sqrt{3},2,\sqrt{5},\sqrt{6},\sqrt{10}, \sqrt{14},3
\sqrt{2} \right\}.
$$
We will now extract the set $A_1$ of square-free radicals that
appear in $A$
$$ A_1 =  \{\sqrt{2},\sqrt{3},\sqrt{5},\sqrt{6},\sqrt{10},\sqrt{14} \}.$$
Our goal is to align everything around real numbers of the form $\alpha = \sqrt{6}k$
for $k \in \mathbb{N}$ which leaves us with $A_2 = A_1\setminus \left\{ \sqrt{6} \right\}$
$$ A_2 =  \{\sqrt{2},\sqrt{3},\sqrt{5},\sqrt{10},\sqrt{14} \}.$$
 Lemma \ref{align} implies that for every arbitrarily small $\varepsilon_1 >0$ there exist infinitely many $\alpha = k
\sqrt{6}$ for $k \in \mathbb{N}$ such that 
$$
\left\{ \alpha \left(1, \sqrt{2}, \sqrt{3},\sqrt{5},\sqrt{10},\sqrt{14}, \right)  \right\}
\in \left(1-\varepsilon,1 \right)^6.  
$$ 
Thus, adding the elements $2$ and $3 \sqrt{2}$ back into the list
$$
 \left\{ \alpha (1,\sqrt{2},\sqrt{3},2,\sqrt{5},\sqrt{10}, \sqrt{14},3
\sqrt{2}) \right\} \in \left(1 - 3 \varepsilon,1 \right)^8.
$$
As customary, we can start the argument with $\varepsilon/3$ and absorb the constant. This means that 
there are infinitely many and arbitrarily large areas $\alpha$ such that for all $p/q \in G \setminus \{3/2\}$
$$
\{ \alpha \sqrt{p q} \} \in (1-\varepsilon,1),
$$
while 
$$
\{ \alpha \sqrt{6 } \} = 0.
$$
Lemma \ref{mimic} then implies the result.
\end{proof}

\subsection{Proof of Theorem \ref{limitset}}
\begin{proof} We will now see how the argument sketched in a special case in the section above can be modified to work in general. We have already seen that $\left\{1 \right\} \subset \Lambda$. Suppose the
statement is false and 
$$ \Lambda \cap S \qquad \mbox{is a finite set} \quad \left\{ \frac{p_1}{q_1}, \dots, \frac{p_n}{q_n}\right\}.$$
Let us then consider the slope $\beta = (p_r+1)/p_r$ where $p_r$ is a
prime number larger any of the $p_i$ or $q_i$. By Lemma
\ref{lem:rational},
the slope $\beta$ is going to be $\gamma$-good at areas $\alpha = k\sqrt{p_r(p_r+1)}$ where $k \in \mathbb{N}$, where $\gamma$ is some fixed
number depending only on $p_r$. The set of rational numbers that can ever possible be $\gamma$-good for infinitely many areas is finite and
we shall denote it by $G$. We now consider the set
$$ A = \left\{ \sqrt{pq}: \frac{p}{q} \in G \right\},$$
write every single element as $\sqrt{p q} = a \sqrt{b}$ with $a \in \mathbb{N}$ and $b \in \mathbb{N}$ and square-free (this decomposition
is unique) and compile, as above, $ A_1$ as the collection of all such $\sqrt{b}$.  Finally, we remove the element that arises from the square-free decomposition
of $\sqrt{p_r(p_r+1)}$. A difference to the proof above is that this element may correspond to more than one slope $p/q$.  Lemma \ref{align} allows, for 
$\varepsilon > 0$, to find infinitely many areas $\alpha = \sqrt{p_r(p_r+1)}k$, $k \in \mathbb{N}$, such that  
$$ \left\{ \alpha \sqrt{p_r(p_r+1)} \right\} = 0~\mbox{(by construction of}~\alpha\mbox{)}$$
 while for all $\sqrt{b} \in A_2$
$$ \left\{ \sqrt{b} \right\} \in \left(1-\varepsilon, 1\right).$$
Moreover, by Lemma \ref{mimic} and the fact that we are dealing with finitely
many rationals, we can pick a sufficiently small $\varepsilon > 0$ such
that for infinitely many $\alpha = \sqrt{p_r (p_r + 1)} k$
$$ \left\{ \sqrt{p q}
\alpha\right\}  \in \left(1-\varepsilon, 1\right),$$
with $\varepsilon > 0$ so small that $p/q$ cannot beat a $\gamma$-good slope --
however, this is only true for $p/q \in G$ whose squarefree-part does not
coincide with the square-free part of $\sqrt{p_r(p_r+1)}$. This square-free
part, however, is bound to contain at least $p_r$ because $p_r$ is prime and
does not divide $p_r + 1$. At the same time, by choice of $p_r$, no element in
the supposedly finite set $\Lambda \cap S$ can be affected. This means that we
have constructed a finite set of slopes, distinct from $\Lambda \cap S$, and an
infinite, unbounded sequence of areas such that the optimal slope is in that new
finite set. By pigeonholing, at least one of the elements has to be in $\Lambda
\cap S$ by applying Lemma \ref{mimic}.
\end{proof}

\subsection{Proof of Theorem \ref{badarea}}
\begin{proof} 
It suffices to show that for every $\gamma >0$ and a sequence of
$(\alpha_k)$ going to infinity $$ \sup_{\beta} N_{\beta}(\alpha_k) \leq
\frac{\alpha_k^2}{2} -\alpha_k +  \frac{\gamma}{2} \alpha_k.$$
In the language of
Lemma \ref{curve}, this means that we are asking for areas such that no
$\gamma-$good slope exists.  Lemma \ref{closetorational} implies that a
$\gamma-$good slope $\beta$ has to be rather close to rational slope $|
\frac{p}{q} - \beta | \lesssim \alpha^{-1}$ satisfying $q \lesssim 10/\gamma$ (in particular, the
set of rational numbers with this property is finite). We know from Lemma \ref{mimic} that
slopes near badly performing rational slopes cannot perform well. However, since there
are only finitely many rational slopes, we can find alignments where not a single one performs
well. These alignments correspond to areas where the optimal slope has to be different from
a number close to one of these few selected $\gamma-$good rational numbers. \end{proof}

\textbf{Acknowledgment.} We are grateful to Rick Laugesen and Shiya Liu for helpful discussions
and to an anonymous referee for many helpful suggestions for improvement and pointing
out the reference to the work of Besicovitch \cite{besicovitch}.


\begin{thebibliography}{2}

\bibitem{AntunesFreitas2012}
P.~R.~S.~{Antunes} and P.~{Freitas}.
\newblock Optimal spectral rectangles and lattice ellipses.
\newblock {\em Proc. Roy. Soc. London Ser. A}, 469(2150), 2012.

\bibitem{AntunesFreitas2016}
P.~R.~S.~{Antunes} and P.~{Freitas}.
\newblock Optimisation of eigenvalues of the Dirichlet Laplacian with a surface
  area restriction.
\newblock {\em Appl. Math. Optim.}, 73(2):313--328, 2016.

\bibitem{AriturkLaugesen2017}
S.~{Ariturk} and R.~S. {Laugesen}.
\newblock {Optimal stretching for lattice points under convex curves.}
\newblock {\em Port. Math.}, 74(2):91--114, 2017.

\bibitem{beck} 
M.~Beck and S.~Robins.
\newblock {\em Computing the continuous discretely.}
\newblock Undergraduate Texts in Mathematics. Springer, New York, 2007.

\bibitem{Berger2015}
A.~Berger.
\newblock The eigenvalues of the Laplacian with Dirichlet boundary condition in
  $\mathbb{R}^2$ are almost never minimized by disks.
\newblock {\em Ann. Global Anal. Geom.}, 47(3):285--304, 2015.

\bibitem{besicovitch} 
A.~S.~Besicovitch. 
\newblock On the linear independence of fractional powers of integers. 
\newblock {\em J. London Math. Soc.}, 15:3--6, 1940.

\bibitem{Boreico2008}
I. Boreico. 
\newblock Linear Independence of Radicals.
\newblock {\em The Harvard College Mathematics Review}, 2(1):87-92, 2008.

\bibitem{BucerFreitas2013}
D.~Bucur and P.~Freitas.
\newblock Asymptotic behaviour of optimal spectral planar domains with fixed
  perimeter.
\newblock {\em J. Math. Phys.}, 54(5):053504, 2013.

\bibitem{Kronecker1884}
L.~{Kronecker}.
\newblock N\"aherungsweise ganzzahlige Aufl\"osung linearer Gleichungen.
\newblock {\em Berl. Ber}, 1179--1193, 1271--1299, 1884.

\bibitem{kuipers} 
L.~Kuipers and H.~Niederreiter.
\newblock {\em Uniform distribution of sequences.}
\newblock Pure and Applied Mathematics. Wiley-Interscience, 1974. 

\bibitem{LaugesenLiu2016}
R.~{Laugesen} and S.~{Liu}.
\newblock {Optimal stretching for lattice points and eigenvalues}.
\newblock {\em ArXiv e-prints} (to appear in {\em Ark. Mat.}), September 2016.

\bibitem{pick}  
G.~Pick. 
\newblock Geometrisches zur Zahlenlehre. 
\newblock {\em Sitzungsberichte des deutschen naturwissenschaftlich-medicinischen
Vereines f\"ur B\"ohmen Lotos in Prag.}, 19:311--319, 1899.

\bibitem{vandenBergBucurGittins2016}
M.~van~den Berg, D.~Bucur, and K.~Gittins.
\newblock Maximising Neumann eigenvalues on rectangles.
\newblock {\em Bull. Lond. Math. Soc.}, 48(5):877--894, 2016.


\bibitem{vandenBergGittins2016}
M.~{van den Berg} and K.~{Gittins}.
\newblock Minimising Dirichlet eigenvalues on cuboids of unit measure.
\newblock {\em Mathematika},  63(2):469--482,  2017.

\bibitem{weyl}
H.~Weyl.
\newblock \"Uber die Gleichverteilung von Zahlen mod. 
\newblock {\em Eins.  Math. Ann.}, 77(3):313--352, 1916.

\end{thebibliography}
\end{document}